\documentclass[12pt]{amsart}
\usepackage{a4wide}

\usepackage{epsfig}
\usepackage{mathtools}
\usepackage{amsmath, mathrsfs}
\usepackage{amssymb}
\usepackage{amsfonts}
\usepackage[centertags]{amsmath}
\usepackage{soul}
\usepackage{graphicx}
\usepackage{xcolor}
\definecolor{deepblue}{RGB}{0,90,170}

\usepackage{amscd}
\usepackage{tikz}
\usepackage{amscd}
\usepackage{graphicx}
\usepackage{color}
\usepackage{verbatim}
\ProvideTextCommandDefault{\cprime}{\tprime}

\newtheorem{theorem}{Theorem}[section]
\newtheorem{lemma}[theorem]{Lemma}
\newtheorem{prop}[theorem]{Proposition}

\renewcommand{\ge}{\geqslant}
\renewcommand{\geq}{\geqslant}
\renewcommand{\le}{\leqslant}
\renewcommand{\leq}{\leqslant}

\newtheorem{remark}[theorem]{Remark}

\numberwithin{equation}{section}

\def \F {\mathcal F}

\def \F {\mathcal F}

\makeatletter
\@namedef{subjclassname@2020}{\textup{2020} Mathematics Subject Classification}
\makeatother

\makeindex

\begin{document}

	\title[Metric Mean Dimension of Amenable Group Actions]{Metric Mean Dimension of Amenable Group Actions: Localization and Non-Uniformity}

	\author{Xinyao He, Guohua Zhang and Ruifeng Zhang}
	
	\address{\vskip 2pt \hskip -12pt Xinyao He}
	
	\address{\hskip -12pt School of Mathematical Sciences, Fudan University, Shanghai 200433, China}
	
	\email{hexy23@m.fudan.edu.cn}

	\address{\vskip 2pt \hskip -12pt Guohua Zhang}
	
	\address{\hskip -12pt School of Mathematical Sciences and Shanghai Center for Mathematical Sciences, Fudan University, Shanghai 200433, China}
	
	\email{chiaths.zhang@gmail.com}

	\address{\vskip 2pt \hskip -12pt Ruifeng Zhang}
	
	\address{\hskip -12pt School of Mathematics, Hefei University of Technology, Hefei 230009, Anhui, China}
	
	\email{rfzhang@hfut.edu.cn}



	\parindent=10pt
	
	\begin{abstract}
		In this paper, we extend Tsukamoto's recent localization formula for metric mean dimension to actions of countable discrete amenable groups, which previously applied only to $\mathbb{R}^k$- and $\mathbb{Z}^k$-actions --- by proving that the global metric mean dimension is characterized by the asymptotic entropy of pointwise $\varepsilon$-stable sets (Theorem \ref{thm6}). To achieve this generalization, we introduce equivalent definitions of the invariant using topological entropy, packing topological entropy, and Bowen's dimensional entropy, respectively. A key technical contribution is our replacement of tiling arguments with Lindenstrauss's combinatorial covering lemma, which enables us to handle the general structure of amenable groups. Furthermore, we resolve all three questions regarding uniformity raised in \cite[{\S 6}]{Yang-Chen-Zhou2024} by constructing counterexamples (Theorem \ref{count} and Proposition \ref{answer}), which demonstrates that the supremum and limit superior in the localization formula cannot generally be interchanged, thereby highlighting the heterogeneous nature of the convergence. These results clarify the uniformity issue and offer insights into the link between local dynamics and global invariants, while our equivalent definitions provide flexible tools for computing metric mean dimension in concrete settings.
	\end{abstract}
	
	\maketitle
	
	\setcounter{tocdepth}{1}
	
	
	\section{Introduction}\label{section1}
	
	Mean dimension is a topological invariant of dynamical systems, introduced by Gromov \cite{Gromov1999} and subsequently developed by Lindenstrauss and Weiss \cite{Lindenstrauss-Weiss2000}. As a natural extension of topological entropy, it was designed to classify systems with infinite entropy by quantifying the ``infinite-dimensionality" of the state space relative to the dynamics. Its metric counterpart, the metric mean dimension \cite{Lindenstrauss-Weiss2000}, was introduced to account for the scale of measurement. For \(\mathbb{Z}\)-actions, these invariants are central to embedding problems and system classification \cite{Gutman-Tsukamoto2020, Lindenstrauss1999-IHESPM}. Recent research has significantly expanded the theory to encompass more general group actions, particularly amenable groups (see for example \cite{Burguet-Shi2022, Chen-Dou-Zheng2022, Dou2017, Gutman2011, Gutman-Lindenstrauss-Tsukamoto2016, Gutman-Qiao-Tsukamoto2019, Hayes2017, Li2013, Li-Liang2018, LiZM2021, Niu2024}).

	Although metric mean dimension provides an upper bound for mean dimension, its calculation remains a significant challenge. For a single transformation, Tsukamoto \cite{Tsukamoto2022} advanced this problem by applying Bowen's foundational idea \cite{R.Bowen1972-TAMS} to clarify the local structure of metric mean dimension. Specifically, he demonstrated that it can be computed via the topological entropy of the
	\(\varepsilon\)-stable sets at individual points --- a surprising result, given that these pointwise
	\(\varepsilon\)-stable sets are typically very small. As an application, Tsukamoto used this approach to simplify his prior proof \cite{Tsukamoto2018} of the mean dimension estimate for the $\mathbb{C}$-action on the space of Brody curves. Subsequently, Yang, Chen, and Zhou \cite{Yang-Chen-Zhou2022, Yang-Chen-Zhou2024} extended Tsukamoto's framework by deriving a formula that expresses metric mean dimension in terms of the packing topological entropy and Bowen's dimensional entropy of these
	\(\varepsilon\)-stable sets.
	
	The computation of metric mean dimension for general group actions is central to the study of a wide range of geometric examples \cite{Gromov1999, Matsuo-Tsukamoto2015, Tsukamoto2018}. It is therefore natural to conjecture that --- as Tsukamoto established in \cite{Tsukamoto2022} for
	$\mathbb{R}^k$- and $\mathbb{Z}^k$-actions --- this invariant can be expressed via the entropy of pointwise \(\varepsilon\)-stable sets.
	
	\smallskip
	
	In this paper, we extend this program to actions of countable discrete amenable groups.
	
	\smallskip
	
	We introduce three equivalent notions of metric mean
	dimension, defined respectively via topological entropy, packing topological entropy, and Bowen's dimensional entropy. Our main result generalizes Tsukamoto's theorem by showing that for amenable group actions, the global metric mean dimension is completely characterized by the asymptotic entropy of these pointwise \(\varepsilon\)-stable sets.
	The proof in the commutative case is essentially an adaptation of Bowen's original argument in \cite{R.Bowen1972-TAMS}, making it a corollary thereof. A key obstruction to generalization is that Tsukamoto's proof relies on the highly regular tiling structure of
	$\mathbb{R}^k$ and $\mathbb{Z}^k$, which does not necessarily exist in general amenable groups. To overcome this, we replace the tiling machinery used in \cite{Tsukamoto2022} with Lindenstrauss's combinatorial covering lemma \cite[Corollary 2.7]{Lindenstrauss2001}, a Vitali-type tool for amenable groups. This enables us to establish a new amenable-group analogue of a key proposition by Bowen \cite[Proposition 2.2]{R.Bowen1972-TAMS}, which forms the technical foundation for our main theorems.
	
	Our localization formula (Theorem~\ref{thm6}) establishes an analogue of Tsukamoto's localization formula for metric mean dimension in the setting of amenable group actions.
	This result, expressed as an equality involving a limsup over scales $\delta \rightarrow 0$ and a supremum over points $x\in X$, naturally raises a subtle analytical question: whether the order of the limsup and the supremum can be interchanged.
	Investigating this (non-)interchangeability is not merely a technical exercise;
	the answer to it illuminates the uniformity of the asymptotic entropy contribution from local stable sets across the space.
	A positive answer would imply a uniform convergence behavior, while a negative answer would reveal an inherent spatial heterogeneity in how local dynamics converge to the global metric mean dimension.
	In fact, we provide a definite negative answer by constructing an explicit counterexample, showing that the pointwise supremum and the scale limit superior in the localization formula cannot be interchanged in general. The construction not only resolves questions raised in \cite[\S 6]{Yang-Chen-Zhou2024}, but also highlights the delicate, non-uniform nature of the localization process: the global metric mean dimension emerges from local data in a manner that is inherently heterogeneous across the phase space.

	\smallskip
	
	We remark that Li \cite{Li2013} has extended the definition of mean dimension to sofic group actions; however, whether our results extend to the sofic setting remains an open question.
	
	\smallskip

	The remainder of this paper is organized as follows. In the next section we state our main results. Section \ref{preli} covers the necessary preliminaries on amenable groups, topological entropy, packing topological entropy, Bowen's dimensional entropy, and also defines the three variants of metric mean dimension that are under study. Section \ref{main} presents the complete proofs of Theorems \ref{thm5} and \ref{thm6}. The core of the argument is Proposition \ref{prop14}, a key technical lemma that bridges local entropy estimates with the global dimension.
	Section \ref{example} is devoted to explicit examples. We first construct a $\mathbb{Z}$-action giving negative answers to all three questions raised in \cite[\S 6]{Yang-Chen-Zhou2024}, and then strengthen this construction to obtain a $\mathbb{Z}$-action in which all of the four quantities $\overline{D}_{\mathrm{int}}$, $\underline{D}_{\mathrm{int}}$, $\overline{D}_{\mathrm{ext}}$ and $\underline{D}_{\mathrm{ext}}$ introduced prior to Theorem \ref{count} can be made pairwise distinct.

	\section{Statements of main results} \label{state}
	
	Let $G$ be a countably infinite discrete group. Denote by $e_{G}$ the unity of the group and $\mathfrak{F}_{G}$ the collection of all finite non-empty subsets of $G$. A sequence $\{H_{n}\}_{n\in\mathbb{N}}\subset \mathfrak{F}_{G}$ is said to be \emph{tempered} if there exists a constant $C>0$, independent of $n\in \mathbb{N}$, such that
	$$\Bigg|\bigcup_{k<n}H_{k}^{-1}H_{n}\Bigg|\leq C|H_{n}|\ \ \ \ \ \ \text{for every}\ n\geq 2,
	$$
	where $|\bullet|$ denotes the cardinality of the set $\bullet$.	The group $G$ is said to be \emph{amenable} if it admits a \emph{F\o lner sequence} $\{F_n\}_{n\in \mathbb{N}}\subset \mathfrak{F}_{G}$, that is,
	$$\lim_{n\rightarrow \infty} \frac{|F_n \cap g F_n|}{|F_n|}= 1\ \ \ \ \ \ \text{for each}\ g\in G.$$
	It is trivial that each amenable group admits a tempered F\o lner sequence $\{F_{n}\}_{n\in\mathbb{N}}$ satisfying
	\begin{equation}
		\lim_{n\to\infty}\dfrac{|F_{n}|}{\log n}=+\infty.
		\label{1.1}
	\end{equation}
	
	By a \emph{$G$-action} $(X,G)$, we mean a nonempty compact metric space $(X, d)$, together with a group $G$ of self-homeomorphisms of $X$ (with $e_{G}$ denoting the
	identity element of $G$).
	\emph{Throughout this paper, we assume $(X, G)$ to be a $G$-action, where $G$ is a countably infinite discrete amenable group
		and $\{F_n\}_{n\in \mathbb{N}}\subset \mathfrak{F}_{G}$ is a F\o lner sequence.}

In the next section we shall introduce three notions of metric mean dimension, each defined via a different entropy measure: the upper metric mean dimension \(\overline{\mathrm{mdim}}_M\) (via topological entropy), the packing upper metric mean dimension \(\overline{\mathrm{mdim}}^P_M\) (via packing topological entropy), and the Bowen upper metric mean dimension \(\overline{\mathrm{mdim}}^B_M\) (via Bowen's dimensional entropy). We emphasize that all quantities involved depend on the choice of the metric $d$ over the state space $X$.

In the remainder of this paper, we shall offer a comprehensive local-global characterization of metric mean dimension for amenable group actions, thereby unifying and extending prior results for $\mathbb{Z}$-actions in \cite{Tsukamoto2022, Yang-Chen-Zhou2022, Yang-Chen-Zhou2024}.

\smallskip
	
	Our first result gives the consistency of these dimensions at both global and local levels.
	
	\begin{theorem}\label{thm5}
		Let $\{F_{n}\}_{n\in\mathbb{N}}$ be a tempered F\o lner sequence satisfying \eqref{1.1}. Then
\begin{equation} \label{f2}
 \overline{\mathrm{mdim}}_{M}(X,\{F_{n}\}_{n\in\mathbb{N}},d) = \max_{x\in X}\overline{\mathrm{mdim}}_{M}(x,\{F_n\}_{n\in\mathbb{N}},d),
\end{equation}
\begin{equation} \label{2.2-2}
\overline{\mathrm{mdim}}_{M}^{P}(X,\{F_{n}\}_{n\in\mathbb{N}},d) =\max_{x\in X}\overline{\mathrm{mdim}}^{P}_{M}(x,\{F_n\}_{n\in\mathbb{N}},d),
\end{equation}
\begin{equation} \label{2.2-3}
\overline{\mathrm{mdim}}_{M}^{B}(X,\{F_{n}\}_{n\in\mathbb{N}},d)=
			\max_{x\in X}\overline{\mathrm{mdim}}^{B}_{M}(x,\{F_n\}_{n\in\mathbb{N}},d),
\end{equation}
\begin{equation*}
\overline{\mathrm{mdim}}_{M}(X,\{F_{n}\}_{n\in\mathbb{N}},d) =
\overline{\mathrm{mdim}}_{M}^{P}(X,\{F_{n}\}_{n\in\mathbb{N}},d) =
\overline{\mathrm{mdim}}_{M}^{B}(X,\{F_{n}\}_{n\in\mathbb{N}},d).
\end{equation*}
	\end{theorem}
	
\begin{remark}
			We can similarly introduce three notions of lower metric mean dimension using $\liminf$, namely \(\underline{\mathrm{mdim}}_M\), \(\underline{\mathrm{mdim}}^P_M\), and \(\underline{\mathrm{mdim}}^B_M\). The same arguments show
			\[
			\underline{\mathrm{mdim}}_{M}^{B}(X,\{F_{n}\}_{n\in\mathbb{N}},d)= \underline{\mathrm{mdim}}_{M}^{P}(X,\{F_{n}\}_{n\in\mathbb{N}},d)=
			\underline{\mathrm{mdim}}_{M}(X,\{F_{n}\}_{n\in\mathbb{N}},d).
			\]
However, it is unknown whether any of them equals
			\[
			\max_{x\in X}\underline{\mathrm{mdim}}_{M}(x,\{F_n\}_{n\in\mathbb{N}},d),\ \max_{x\in X}\underline{\mathrm{mdim}}^{P}_{M}(x,\{F_n\}_{n\in\mathbb{N}},d), \ \text{or}\ \max_{x\in X}\underline{\mathrm{mdim}}^{B}_{M}(x,\{F_n\}_{n\in\mathbb{N}},d).
			\]
	\end{remark}

	Our second result provides a precise formula for the global metric mean dimension in terms of various entropies of local stable sets.\footnote{
		The same arguments show that the conclusion of Theorem \ref{thm6} remains valid when
		$\overline{\mathrm{mdim}}$ and $\limsup$
		are replaced everywhere by $\underline{\mathrm{mdim}}$ and $\liminf$, respectively; moreover, we can also consider
		alternatively
		the \(\varepsilon\)-stable set $\Gamma_{\varepsilon}^{\left\{F_n\right\}_{n \in \mathbb{N}}}(x)$.
		In fact \eqref{f1} and its variant (for $\underline{\mathrm{mdim}}$ and $\liminf$)
		have been proved in the setting of
		the F{\o}lner sequence
		$F_n=\{0, 1, \cdots,n-1\},\ n\in\mathbb N$
		for a single transformation in \cite{Tsukamoto2022}.
	}
	
	\begin{theorem}\label{thm6}
		Let $\{F_{n}\}_{n\in\mathbb{N}}$ be a tempered F\o lner sequence satisfying \eqref{1.1}. Then
		\begin{eqnarray} \label{f1}
			\overline{\mathrm{mdim}}_{M}(X,\{F_n\}_{n\in\mathbb{N}},d) & = & \limsup_{\delta\to 0}\sup_{x\in X}\frac{h_{\mathrm{top}}(\Gamma_{\varepsilon}(x),\delta,\{F_{n}\}_{n\in\mathbb{N}},d)}{\log\frac{1}{\delta}} \\
			& = & \limsup_{\delta\to 0}\sup_{x\in X}\frac{h_{\mathrm{top}}^{P}(\Gamma_{\varepsilon}(x),\delta,\{F_{n}\}_{n\in\mathbb{N}},d)}{\log\frac{1}{\delta}}
			\nonumber \\
			& = & \limsup_{\delta\to 0}\sup_{x\in X}\frac{h_{\mathrm{top}}^{B}(\Gamma_{\varepsilon}(x),\delta,\{F_{n}\}_{n\in\mathbb{N}},d)}{\log\frac{1}{\delta}} \nonumber
		\end{eqnarray}
		for every $\varepsilon>0$, where \(\Gamma_\varepsilon(x) = \{ y \in X: d(gx, gy) \le \varepsilon,\ \forall g \in G \}\) is the \emph{\(\varepsilon\)-stable set} of \(x\).
	\end{theorem}
	
	\begin{remark}
		It was also considered in \cite{Dou-Wang-Zhang2025} the following version of \(\varepsilon\)-stable set of \(x\):
		\begin{equation} \label{stable}
			\Gamma_{\varepsilon}^{\left\{F_n\right\}_{n \in \mathbb{N}}}(x)=\bigcap_{n \in \mathbb{N}} \{ y \in X: d(gx, gy) \le \varepsilon,\ \forall g \in F_n \} \supset \Gamma_{\varepsilon}(x).
		\end{equation}
		From the definition of the upper metric mean dimension presented in the next section, the conclusion in Theorem \ref{thm6} is trivially seen to still hold if we replace \(\Gamma_\varepsilon(x)\) by \(\Gamma_{\varepsilon}^{\left\{F_n\right\}_{n \in \mathbb{N}}}(x)\).
	\end{remark}

	\begin{remark} \label{tem}
		The proofs of Theorems \ref{thm5} and \ref{thm6} rely crucially on the Lindenstrauss covering lemma, which accounts for the temperedness requirement on the F\o lner sequence. As shown by Theorem \ref{without-tem}, using a standard technique to pass to subsequences, we can nevertheless extend \eqref{f2}, \eqref{2.2-2}, \eqref{2.2-3} and \eqref{f1} to arbitrary F\o lner sequences.
	\end{remark}

	For simplicity, we introduce the following notations: for every $\varepsilon>0$,
	\[
	\begin{aligned}
		\overline{D}_{\mathrm{int}}(X,\varepsilon,\{F_n\}_{n\in\mathbb{N}},d)&:=\limsup_{\delta\to 0}\sup_{x\in X}
		\frac{h_{\mathrm{top}}(\Gamma_{\varepsilon}^{\{F_n\}_{n\in\mathbb{N}}}(x),\delta,\{F_n\}_{n\in\mathbb N},d)}
		{\log \frac{1}{\delta}},\\
		\underline{D}_{\mathrm{int}}(X,\varepsilon,\{F_n\}_{n\in\mathbb{N}},d)&:=\liminf_{\delta\to 0}\sup_{x\in X}
		\frac{h_{\mathrm{top}}(\Gamma_{\varepsilon}^{\{F_n\}_{n\in\mathbb{N}}}(x),\delta,\{F_n\}_{n\in\mathbb N},d)}
		{\log \frac{1}{\delta}},\\
		\overline{D}_{\mathrm{ext}}(X,\varepsilon,\{F_n\}_{n\in\mathbb{N}},d)&:=\sup_{x\in X}\limsup_{\delta\to 0}
		\frac{h_{\mathrm{top}}(\Gamma_{\varepsilon}^{\{F_n\}_{n\in\mathbb{N}}}(x),\delta,\{F_n\}_{n\in\mathbb N},d)}
		{\log \frac{1}{\delta}},\\
		\underline{D}_{\mathrm{ext}}(X,\varepsilon,\{F_n\}_{n\in\mathbb{N}},d)&:=\sup_{x\in X}\liminf_{\delta\to 0}
		\frac{h_{\mathrm{top}}(\Gamma_{\varepsilon}^{\{F_n\}_{n\in\mathbb{N}}}(x),\delta,\{F_n\}_{n\in\mathbb N},d)}
		{\log \frac{1}{\delta}}.
	\end{aligned}
	\]
	The subscripts int and ext indicate whether the supremum over $x\in X$
	is taken internally or externally with respect to the limit in $\delta$.
	In particular,
	\begin{equation} \label{relation}
		\begin{aligned}
			\overline{D}_{\mathrm{int}}(X,\varepsilon,\{F_n\}_{n\in\mathbb{N}},d) &\ge \max \{
			\underline{D}_{\mathrm{int}}(X,\varepsilon,\{F_n\}_{n\in\mathbb{N}},d), \overline{D}_{\mathrm{ext}}(X,\varepsilon,\{F_n\}_{n\in\mathbb{N}},d)\} \\&\ge 	\underline{D}_{\mathrm{ext}}(X,\varepsilon,\{F_n\}_{n\in\mathbb{N}},d),
		\end{aligned}
	\end{equation}
	and
	\[
	\begin{aligned}
		&\overline{D}_{\mathrm{int}}(X,\varepsilon,\{F_n\}_{n\in\mathbb{N}},d) = \overline{\mathrm{mdim}}_{M}(X,\{F_n\}_{n\in\mathbb{N}},d),
		\\&
		\underline{D}_{\mathrm{int}}(X,\varepsilon,\{F_n\}_{n\in\mathbb{N}},d) = \underline{\mathrm{mdim}}_{M}(X,\{F_n\}_{n\in\mathbb{N}},d).
	\end{aligned}
	\]

Our last result suggests that the contribution of local dynamics to the global metric mean dimension varies from point to point.

	\begin{theorem}\label{count}
		For any $0\le a\leq b$, there exist a $\mathbb{Z}$-action $(X, \mathbb{Z})$ with a compatible metric $d$ such that, for every $\varepsilon\in(0,1)$ and the F{\o}lner sequence
		$F_n=\{0, 1, \cdots,n-1\},\ n\in\mathbb N$,
		\[
		\begin{aligned}
			&\overline{D}_{\mathrm{int}}(X,\varepsilon,\{F_n\}_{n\in\mathbb{N}},d)=1+b,\quad
			\underline{D}_{\mathrm{int}}(X,\varepsilon,\{F_n\}_{n\in\mathbb{N}},d)=1+a,\quad
			\\&\overline{D}_{\mathrm{ext}}(X,\varepsilon,\{F_n\}_{n\in\mathbb{N}},d)=b,\quad
			\quad\,\,\,\underline{D}_{\mathrm{ext}}(X,\varepsilon,\{F_n\}_{n\in\mathbb{N}},d)=a.
		\end{aligned}
		\]
		In particular, by choosing suitable $a$ and $b$, one obtains a $\mathbb{Z}$-action for which all of these four quantities are pairwise distinct for every $\varepsilon\in(0,1)$.
	\end{theorem}

\begin{remark}
Theorem \ref{count} shows that, the order of the limsup over scales $\delta \rightarrow 0$ and the supremum over points $x\in X$ in Theorem \ref{thm6} can not be interchanged. Indeed, it demonstrates, firstly, that in the localization formula the supremum and limit superior of topological entropy cannot be interchanged; by Theorem \ref{thm6} and Proposition \ref{prop20} (3) below,
the same conclusion holds for packing topological entropy and Bowen's dimensional entropy. Moreover, the order of the liminf and the supremum for topological entropy can not be interchanged, if $\overline{\mathrm{mdim}}$ and $\limsup$ in \eqref{f1}
	are replaced by $\underline{\mathrm{mdim}}$ and $\liminf$, respectively. In particular, this answers \cite[\S 6, Question (1)]{Yang-Chen-Zhou2024} in the negative.
	We remark that, in fact, the same construction also applies to the preimage
	and asymptotic stable-set versions as shown by Proposition \ref{answer}; hence we also resolve
\cite[\S 6, Questions~(2) and (3)]{Yang-Chen-Zhou2024} in the negative.
\end{remark}

	\section{Preliminaries} \label{preli}
	
	In this section let us recall necessary preliminaries.
	
	\subsection{Basic facts on amenable groups and their actions}\
	
	We begin by recalling essential tools from the entropy theory of amenable group actions.
	
	A key tool in the entropy theory of amenable group actions is the well-known Ornstein-Weiss convergence lemma \cite{Ornstein-Weiss1987}. The version presented here follows \cite[1.3.1]{Gromov1999}.

	\begin{lemma} \label{OW-Lemma}
		Let $f: \mathfrak{F}_G \rightarrow \mathbb{R}_+$ be an invariant subadditive function, that is,
		$f (F g)= f (F)$ and $f (E\cup F)\le f (E)+ f (F)$ whenever $E, F\in \mathfrak{F}_G$ and $g\in G$.
		Then the limit $\lim\limits_{n\to \infty} \frac{f (F_n)}{|F_n|}$ exists and is independent of the selection of the F\o lner sequence $\{F_{n}\}_{n\in\mathbb{N}}$.
	\end{lemma}
	
	The following combinatorial covering lemma due to Lindenstrauss \cite[Corollary 2.7]{Lindenstrauss2001} is central to our proofs.
	Recall that a family \(\mathcal{F} \subset \mathfrak{F}_G\) is \emph{\(\gamma\)-disjoint} (\(0 \le \gamma < 1\)) if for each \(A \in \mathcal{F}\) there exists \(A' \subset A\) with \(|A'| \ge (1-\gamma)|A|\) such that the sets \(A'\) are pairwise disjoint for different \(A \in \mathcal{F}\). In
	particular, a 0-disjoint family means equivalently a disjoint family.
	
	\begin{lemma}\label{coveringlem}
		For any \(\delta \in (0, 10^{-4})\), \(C > 0\), and \(D \in \mathfrak{F}_G\), there exists \(M \in \mathbb{N}\) large enough (depending on \(\delta, C, D\)) such that: Once \(H \in \mathfrak{F}_G\) and families \(\{H_{i,j}: i=1,\cdots,M, j=1,\cdots,N_i\}\) and \(\{A_{i,j}: i=1,\cdots,M, j=1,\cdots,N_i\}\) satisfy:
		\begin{enumerate}

			\item
			Letting \(H_{i,*} = \bigcup_{j=1}^{N_i} H_{i,j}\), we have for all \(i=1,\cdots,M, k=2,\cdots,N_i\),
			\[
			\left| \bigcup_{k'<k} H_{i,k'}^{-1} H_{i,k} \right| \le C |H_{i,k}|,
			\]
			and for all \(i=2,\cdots,M, k=1,\cdots,N_i\),
			\[
			\left| \bigcup_{i'<i} D H_{i',*}^{-1} H_{i,k} \right| \le (1+\delta) |H_{i,k}|.
			\]
			
			\item For all \(i=1,\cdots,M, j=1,\cdots,N_i\), we have \(H_{i,j} A_{i,j} \subset H\).
		\end{enumerate}
		Define \(A_{i,*} = \bigcup_{j=1}^{N_i} A_{i,j}\) and \(\alpha = \frac{\min_{1\le i \le M} |D A_{i,*}|}{|H|}\). Then the collection \(\{H_{i,j}a: a \in A_{i,j}, i=1,\cdots,M, j=1,\cdots,N_i\}\) admits a \(10\delta^{\frac{1}{4}}\)-disjoint subfamily \(\mathcal{F}\) satisfying \(|\cup \mathcal{F}| \ge (\alpha - \delta^{\frac{1}{4}}) |H|\).
	\end{lemma}
	
	We also employ the following result from \cite[Lemma 5.3]{Dou-Wang-Zhang2025}, which is a variant of a combinatorial lemma due to Lindenstrauss \cite[Lemma 4.2]{Lindenstrauss2001}.

	\begin{lemma}\label{combilem}
		Assume that the F\o lner sequence $\{F_{n}\}_{n\in\mathbb{N}}$ satisfies \eqref{1.1}. Then, for any $\eta>0$, there exists $N=N(\eta)\in \mathbb{N}$ such that for all indices $N\leq n_{1}<\cdots<n_{r}$, the number of $\frac{1}{2}$-disjoint sub-collections of $\{F_{n_{i}}a:F_{n_{i}}a\subset F_{n},i=1,\cdots,r\}$ is at most $2^{\eta|F_{n}|}$ for every sufficiently large $n\in \mathbb{N}$ (depending on $n_{1},\cdots,n_{r}$).
	\end{lemma}

	\subsection{Various entropies of subsets}\
	
	In this subsection, let us recall various entropies of subsets used later.
	
	Denote by $\mathfrak{C}_{X}^{o}$ the set of all finite open covers of $X$. For every $\varepsilon>0$ and $x\in X$, let $B_{\varepsilon}^{d}(x)$ and $O_{\varepsilon}^{d}(x)$,
	respectively, denote the closed and open ball with center $x$ and radius $\varepsilon$, which are also
	denoted by $B_{\varepsilon}(x)$ and $O_{\varepsilon}(x)$ respectively if there is no ambiguity. Furthermore, for any $F\in\mathfrak{F}_{G}$, let $d_{F}(x,y)=\max_{g\in F}d(gx,gy)$. We denote for simplicity $B_{\varepsilon}^{F}(x)$ and $O_{\varepsilon}^{F}(x)$, the closed and open ball of $X$ with center $x$ and radius $\varepsilon$ under the metric $d_{F}$.
	
	\subsubsection{Topological entropy of subsets}\
	
	For a nonempty subset \(E \subset X\), its \emph{topological entropy} $h_{\mathrm{top}}(E, \{F_n\}_{n\in \mathbb{N}})$  is defined as
	\[
	\begin{aligned}
		&h_{\mathrm{top}}(E, \{F_n\}_{n\in \mathbb{N}}) = \sup_{\mathcal{U} \in \mathfrak{C}_X^o} h_{\mathrm{top}}(E, \mathcal{U},\{F_n\}_{n\in \mathbb{N}}),
		\\&h_{\mathrm{top}}(E, \mathcal{U},\{F_n\}_{n\in \mathbb{N}})=\limsup_{n\to\infty} \frac{1}{|F_n|} \log N(\mathcal{U}^{F_n}, E),
	\end{aligned}
	\]
	where \(N(\mathcal{U}^{F_n}, E)\) is the minimal cardinality of a subcover of \(\mathcal{U}^{F_n} := \bigvee_{g \in F_n} g^{-1}\mathcal{U}\) that covers \(E\). An equivalent definition uses separated or spanning sets. For \(\xi > 0\), \(F \in \mathfrak{F}_G\), and \(K \subset X\),
	a set \(H \subset K\) is called \emph{\((F, \xi)\)-separated} if \(d_F(x,y) > \xi\) for all distinct \(x, y \in H\).
	Let \(s_F(K, \xi, d)\) denote the maximal cardinality of an \((F,\xi)\)-separated subset of \(K\), and set
	\begin{equation} \label{topo}
		h_{\text{top}}(K,\delta,\{F_{n}\}_{n\in\mathbb{N}},d) = \limsup_{n\to\infty} \frac{1}{|F_n|} \log s_{F_n}(K, \delta, d).
	\end{equation}
	Then
	\[h_{\mathrm{top}}(K, \{F_n\}_{n\in \mathbb{N}}) = \lim_{\delta \to 0} h_{\text{top}}(K,\delta,\{F_{n}\}_{n\in\mathbb{N}},d).\]
	Similarly, for \(\xi > 0\), \(F \in \mathfrak{F}_G\), and \(K \subset X\),
	a set \(E \subset X\) is called an \emph{\((F, \xi)\)-spanning set} for \(K\) if for every \(x \in K\) there exists \(y \in E\) such that \(d_F(x,y) \le \xi\).
	If let \(r_F(K, \xi, d)\) denote the minimal cardinality of an \((F,\xi)\)-spanning set for \(K\), then it is trivial to obtain
	\begin{equation} \label{se-sp}
		r_F(K, \xi, d) \leq s_F(K, \xi, d) \leq r_F(K, \frac{\xi}{2}, d)
	\end{equation}
	and hence
	\[
	h_{\mathrm{top}}(K, \{F_n\}_{n\in \mathbb{N}}) = \lim_{\delta \to 0} \limsup_{n\to\infty} \frac{1}{|F_n|} \log r_{F_n}(K, \delta, d).
	\]
	
	\begin{remark}
		Using the straightforward observation in \eqref{se-sp}, if in \eqref{topo} we instead adopt
		\[
		h_{\mathrm{top}}\bigl(K,\delta,\{F_n\}_{n\in\mathbb{N}},d\bigr)=\limsup_{n\to\infty}\frac{1}{|F_n|}\log r_{F_n}(K,\delta,d),
		\]
		then the main results of the paper continue to hold. Nevertheless, for the sake of simplicity, we will work exclusively with the definition given in
		\eqref{topo}.
	\end{remark}
	
	\subsubsection{Packing topological entropy of subsets}\
	
	Following \cite{Dou-Zheng-Zhou2023}, for \(\emptyset \neq Z \subset X\) and \(\lambda, \varepsilon, \delta > 0\), define
	
	\[
	P_{\lambda, \varepsilon}(\delta, Z, \{F_n\}_{n\in \mathbb{N}}, d) = \sup \sum_{i \in I} \exp(-\lambda |F_{n_i}|),
	\]
	where the supremum is taken over all finite or countable disjoint families \(\{B^{F_{n_i}}_\delta(x_i)\}_{i \in I}\) with \(x_i \in Z\) and \(|F_{n_i}| \ge \frac{1}{\varepsilon}\) for all \(i\in I\).
	Note that $G$ is assumed to be a countably infinite discrete amenable group and \(\{F_n\}_{n\in \mathbb{N}}\) is a F\o lner sequence,
	\(|F_{n}| \ge \frac{1}{\varepsilon}\) once $n\in \mathbb{N}$ is large enough. We also set $P_{\lambda, \varepsilon}(\delta, \emptyset, \{F_n\}_{n\in \mathbb{N}}, d) = 0$ by convention.
	It is easy to see that $P_{\lambda, \varepsilon}(\delta, Z, \{F_n\}_{n\in \mathbb{N}}, d)$ does not increase as $\varepsilon$ decreases to $0$. Let
	\[P_{\lambda}(\delta, Z, \{F_n\}_{n\in \mathbb{N}}, d) = \lim_{\varepsilon \to 0} P_{\lambda, \varepsilon}(\delta, Z, \{F_n\}_{n\in \mathbb{N}}, d).\]
	Further set
	\[
	\widetilde{P}_{\lambda}(\delta, Z, \{F_n\}_{n\in \mathbb{N}}, d) = \inf\left\{ \sum_{i=1}^\infty P_{\lambda}(\delta, Z_i, \{F_n\}_{n\in \mathbb{N}}, d): Z \subset \bigcup_{i=1}^\infty Z_i \right\}.
	\]
	In particular, $\widetilde{P}_{\lambda}(\delta, Z, \{F_n\}_{n\in \mathbb{N}}, d)\notin\{0,+\infty\}$ for at most one $\lambda>0$, so we may define
	\[
	\begin{aligned}
		h_{\mathrm{top}}^{P}(Z,\delta,\{F_{n}\}_{n\in\mathbb{N}},d)&=\inf\{\lambda>0:\widetilde{P}_{\lambda}(\delta,Z,\{F_{n}\}_{n\in\mathbb{N}},d)=0\}
		\\&=\sup\{\lambda>0:\widetilde{P}_{\lambda}(\delta,Z,\{F_{n}\}_{n\in\mathbb{N}},d)=+\infty\},
	\end{aligned}
	\]
	where we set $\inf \emptyset=+\infty$ by convention.
	In particular,
	$\widetilde{P}_{\lambda}(\delta,Z,\{F_{n}\}_{n\in\mathbb{N}},d)=0$ for all $\lambda> h_{\mathrm{top}}^{P}(Z,\delta,\{F_{n}\}_{n\in\mathbb{N}},d)$.
	It is not hard to see that $h_{\mathrm{top}}^{P}(Z,\delta,\{F_{n}\}_{n\in\mathbb{N}},d)$ increases when $\delta$ decreases to $0$.
	The \emph{packing topological entropy} of \(Z\) is defined as
	\[
	h^P_{\mathrm{top}}(Z, \{F_n\}_{n\in \mathbb{N}}) = \lim_{\delta \to 0} h^P_{\mathrm{top}}(Z, \delta, \{F_n\}_{n\in \mathbb{N}}, d) = \sup_{\delta > 0} h^P_{\mathrm{top}}(Z, \delta, \{F_n\}_{n\in \mathbb{N}}, d),
	\]
	which is independent of the metric $d$ though $h^P_{\mathrm{top}}(Z, \delta, \{F_n\}_{n\in \mathbb{N}}, d)$ does depend on $d$.
	
	\subsubsection{Bowen's dimensional entropy of subsets}\
	
	Now let us recall from \cite{Zheng-Chen2016} the definition of Bowen's dimensional entropy of subsets for amenable group actions. Let \(\mathcal{U} \in \mathfrak{C}_X^o\), \(\emptyset \neq Z \subset X\), and \(\lambda, \varepsilon > 0\). Define
	\[
	F_{\lambda, \varepsilon}(\mathcal{U}, Z, \{F_n\}_{n\in \mathbb{N}}) = \inf \sum_{i \in I} \exp(-\lambda |F_{n_i}|),
	\]
	where the infimum is taken over all finite or countable families \(\{B_i\}_{i \in I}\) covering \(Z\) with each \(B_i \in \mathcal{U}^{F_{n_i}}\) for some \(|F_{n_i}| \ge \frac{1}{\varepsilon}\).
	Set $F_{\lambda,\varepsilon}(\mathcal{U},\emptyset,\{F_{n}\}_{n\in\mathbb{N}})=0$ by convention, and put
	\[
	F_{\lambda}(\mathcal{U},Z,\{F_{n}\}_{n\in\mathbb{N}})=\lim_{\varepsilon\to 0}F_{\lambda,\varepsilon}(\mathcal{U},Z,\{F_{n}\}_{n\in\mathbb{N}})=\sup_{\varepsilon>0}F_{\lambda,\varepsilon}(\mathcal{U},Z,\{F_{n}\}_{n\in\mathbb{N}}).
	\]
	Similarly we may define
	\[
	h_{\mathrm{dim}}(Z,\mathcal{U},\{F_{n}\}_{n\in\mathbb{N}})=\inf\{\lambda>0:F_{\lambda}(\mathcal{U},Z,\{F_{n}\}_{n\in\mathbb{N}})=0\}.
	\]
	In particular,
	$F_{\lambda}(\mathcal{U},Z,\{F_{n}\}_{n\in\mathbb{N}})=0$ for all $\lambda> h_{\mathrm{dim}}(Z,\mathcal{U},\{F_{n}\}_{n\in\mathbb{N}})$.
	The \emph{Bowen's dimensional entropy} of \(Z\) is
	defined as
	\[
	h_{\mathrm{dim}}(Z, \{F_n\}_{n\in \mathbb{N}}) = \sup_{\mathcal{U} \in \mathfrak{C}_X^o} h_{\mathrm{dim}}(Z, \mathcal{U}, \{F_n\}_{n\in \mathbb{N}}).
	\]
	
	Similarly, a metric version can be given for Bowen's dimensional entropy. Let $\delta>0$. For any $\emptyset\neq Z\subset X$ and every $\lambda,\varepsilon>0$ set
	\[
	F_{\lambda,\varepsilon}(\delta,Z,\{F_{n}\}_{n\in\mathbb{N}},d)=\inf \sum_{i\in I}\exp(-\lambda|F_{n_{i}}|),
	\]
	where the infimum is taken over all finite or countable families $\{O_{\delta}^{F_{n_{i}}}(x_{i})\}_{i\in I}$ which cover $Z$ and $|F_{n_{i}}|\geq \frac{1}{\varepsilon}$  for each $i\in I$.
	Let
	\[F_{\lambda}(\delta, Z, \{F_n\}_{n\in\mathbb{N}}, d) = \lim_{\varepsilon \to 0} F_{\lambda, \varepsilon}(\delta, Z, \{F_n\}_{n\in\mathbb{N}}, d)
	\]
	and set \(F_{\lambda}(\delta, \emptyset, \{F_n\}_{n\in\mathbb{N}}, d) = 0\)  by convention. Define
	\[
	h^B_{\mathrm{top}}(Z, \delta, \{F_n\}_{n\in\mathbb{N}}, d) = \inf\{\lambda > 0: F_{\lambda}(\delta, Z, \{F_n\}_{n\in\mathbb{N}}, d) = 0\}.
	\]
	It can be shown by easy arguments that
	\[
	h_{\mathrm{dim}}(Z, \{F_n\}_{n\in \mathbb{N}}) = \lim_{\delta \to 0} h^B_{\mathrm{top}}(Z, \delta, \{F_n\}_{n\in\mathbb{N}}, d) = \sup_{\delta > 0} h^B_{\mathrm{top}}(Z, \delta, \{F_n\}_{n\in\mathbb{N}}, d),
	\]
	so the limit is also independent of the metric \(d\), providing an alternative definition.
	
	\subsection{Metric mean dimension of subsets}\
	
	Building on the entropy notions above, we define three types of upper metric mean dimension, natural analogues for amenable group actions of those in \cite{Yang-Chen-Zhou2022, Yang-Chen-Zhou2024}.
	
	For a nonempty subset \(Z \subset X\), define
	\[
	\begin{aligned}
		\overline{\mathrm{mdim}}_M(Z, \{F_n\}_{n\in \mathbb{N}}, d) &= \limsup_{\delta \to 0} \frac{h_{\mathrm{top}}(Z, \delta, \{F_n\}_{n\in \mathbb{N}}, d)}{\log \frac{1}{\delta}}, \\
		\overline{\mathrm{mdim}}^P_M(Z, \{F_n\}_{n\in \mathbb{N}}, d) &= \limsup_{\delta \to 0} \frac{h^P_{\mathrm{top}}(Z, \delta, \{F_n\}_{n\in \mathbb{N}}, d)}{\log \frac{1}{\delta}}, \\
		\overline{\mathrm{mdim}}^B_M(Z, \{F_n\}_{n\in \mathbb{N}}, d) &= \limsup_{\delta \to 0} \frac{h^B_{\mathrm{top}}(Z, \delta, \{F_n\}_{n\in \mathbb{N}}, d)}{\log \frac{1}{\delta}}.
	\end{aligned}
	\]
	We call them the \emph{upper metric mean dimension}, \emph{packing upper metric mean dimension}, and \emph{Bowen upper metric mean dimension} of \(Z\), respectively.\footnote{We can introduce similarly the \emph{lower metric mean dimension}, \emph{packing lower metric mean dimension}, and \emph{Bowen lower metric mean dimension} of \(Z\), respectively, when
		$\overline{\mathrm{mdim}}$ and $\limsup$
		are replaced everywhere by $\underline{\mathrm{mdim}}$ and $\liminf$, respectively.
	}
	
	Following \cite{Chen-Dou-Zheng2022}, we could provide an equivalent formulation of the upper metric mean dimension in terms of open covers.
	For each $\xi>0$, set
	\[
	\#(Z,\xi, d)=\min\{|\mathcal{U}|:\mathcal{U}\text{ is a finite open cover of $Z$ with }\mathrm{diam}(\mathcal{U},d)<\xi\},
	\]
	where $\mathrm{diam}(\mathcal{U},d)$ denotes the diameter of the family $\mathcal{U}$.
	It is direct to see that
	\begin{equation}\label{opencover}
		r_{F}(Z,\xi,d)\leq \#(Z,\xi, d_{F})\leq r_{F}(Z,\frac{\xi}{3},d)
	\end{equation}
	for any $F\in\mathfrak{F}_{G}$.
	Combining (\ref{se-sp}) and (\ref{opencover}) one has
	\[
	\overline{\mathrm{mdim}}_M(Z, \{F_n\}_{n\in \mathbb{N}}, d)=\limsup_{\delta \to 0}\frac{1}{\log \frac{1}{\delta}}\limsup_{n\to\infty}\frac{1}{|F_{n}|}\log \#(Z,\delta,d_{F_n}),
	\]
	\[
	\underline{\mathrm{mdim}}_M(Z, \{F_n\}_{n\in \mathbb{N}}, d)=\liminf_{\delta \to 0}\frac{1}{\log \frac{1}{\delta}}\limsup_{n\to\infty}\frac{1}{|F_{n}|}\log \#(Z,\delta,d_{F_n}).
	\]
	
	The definitions above appear to depend on the choice of the F\o lner sequence. However:
	
	\begin{prop} \label{rmk-independent}
		Both $\overline{\mathrm{mdim}}_M(X, \{F_n\}_{n\in \mathbb{N}}, d)$ and $\underline{\mathrm{mdim}}_M(X, \{F_n\}_{n\in \mathbb{N}}, d)$
		are independent of the choice of the F\o lner sequence $\{F_n\}_{n\in \mathbb{N}}$.
	\end{prop}

	\begin{proof}
		Firstly we claim that for all \(\xi>0\) the following
		is an invariant subadditive function
		\[
		\mathfrak{F}_{G}\ni F\mapsto \log \#(X,\xi,d_{F}).
		\]
		Take arbitrarily \(F\in \mathfrak{F}_{G}\) and \(g\in G\). Let \(\mathcal U\in\mathfrak C_X^o\) satisfy
		\[
		|\mathcal U|=\#(X,\xi,d_F)
		\quad\text{and}\quad
		\mathrm{diam}(\mathcal U,d_F)<\xi.
		\]
		Then \(g^{-1}\mathcal U:=\{g^{-1}U:U\in\mathcal U\}\in\mathfrak C_X^o\) and
		\(\mathrm{diam}(g^{-1}U,d_{Fg})=\mathrm{diam}(U,d_F)<\xi
		\) for every \(g^{-1}U\in g^{-1}\mathcal U\).
		Hence
		\(
		\#(X,\xi,d_{Fg})\le |g^{-1}\mathcal U|=|\mathcal U|=\#(X,\xi,d_F).
		\)
		Applying this with \(F\) replaced by \(F g\) and $g$ replaced by \(g^{-1}\), we obtain the reverse inequality, and finally
		\(
		\#(X,\xi,d_{Fg})=\#(X,\xi,d_F).
		\)
		This proves the invariance of the function \(F\mapsto \log \#(X,\xi,d_F)\).

		Now take arbitrarily \(F_1,F_2\in\mathfrak F_G\). For each \(i=1,2\), we choose \(\mathcal U_i\in\mathfrak C_X^o\) such that
		\[
		|\mathcal U_i|=\#(X,\xi,d_{F_i})
		\quad\text{and}\quad
		\mathrm{diam}(\mathcal U_i,d_{F_i})<\xi.
		\]
		Then \(\mathcal U_1\vee\mathcal U_2\in\mathfrak C_X^o\) satisfies
		\(
		\mathrm{diam}(\mathcal U_1\vee\mathcal U_2,d_{F_1\cup F_2})<\xi.
		\)
		It follows that
		\[
		\#(X,\xi,d_{F_1\cup F_2})
		\le |\mathcal U_1\vee\mathcal U_2|
		\le |\mathcal U_1|\cdot\,|\mathcal U_2|
		= \#(X,\xi,d_{F_1})\cdot \#(X,\xi,d_{F_2}).
		\]
		This proves the subadditivity of the function \(F\mapsto \log {\#(X,\xi,d_F)}\).
		
		Now applying Lemma \ref{OW-Lemma} we conclude that for any $\xi>0$ the limit
		\[
		\lim_{n\to\infty}\frac{1}{|F_{n}|}\log \#(X,\xi,d_{F_n})
		\]
		exists and is independent of the selection of the F\o lner sequence $\{F_n\}_{n\in \mathbb{N}}$, from which the conclusion follows directly.
	\end{proof}

	Given a point \(x \in X\), its \emph{local upper metric mean dimension} is defined by
	\[
	\overline{\mathrm{mdim}}_M(x, \{F_n\}_{n\in \mathbb{N}}, d) = \inf \overline{\mathrm{mdim}}_M(K, \{F_n\}_{n\in \mathbb{N}}, d),
	\]
	where the infimum is taken over all closed neighborhoods \(K\) of \(x\). The local versions \(\overline{\mathrm{mdim}}^P_M(x, \{F_n\}_{n\in \mathbb{N}}, d)\) and \(\overline{\mathrm{mdim}}^B_M(x, \{F_n\}_{n\in \mathbb{N}}, d)\) are defined analogously.
	We remark that all of these quantities involved depend on the choice of the metric \(d\).

	We note that these local versions of upper metric mean dimension were first defined in \cite{Yang-Chen-Zhou2024} for a single transformation. The idea behind them can in fact be traced back further to the concept of an entropy point from a local viewpoint introduced in \cite{Ye-Zhang2007}.
	
	\smallskip
	
	The following captures a basic relationship among these three dimensions. The case of a single transformation is treated in \cite[Proposition 2.6]{Yang-Chen-Zhou2024}, while certain aspects for amenable group actions are covered in \cite{Dou-Zheng-Zhou2023}.
	We supply a proof here for completeness.\footnote{Proposition \ref{prop20} (1) and (3) remain valid if we consider those three lower metric mean dimensions.}
	
	\begin{prop}\label{prop20}
		Let \(Z_1, Z_2, \cdots, Z_m, Z \subset X\).
		\begin{enumerate}
			
			\item
			If \(Z_1 \subset Z_2\), then \(\overline{\mathrm{mdim}}_M(Z_1, \{F_n\}_{n\in \mathbb{N}}, d) \le \overline{\mathrm{mdim}}_M(Z_2, \{F_n\}_{n\in \mathbb{N}}, d)\). The same conclusion also holds for \(\overline{\mathrm{mdim}}^P_M\) and \(\overline{\mathrm{mdim}}^B_M\).
			
			\item
			If \(Z = \bigcup_{i=1}^m Z_i\), then \(\overline{\mathrm{mdim}}_M(Z, \{F_n\}_{n\in \mathbb{N}}, d) = \max_{1\le i \le m} \overline{\mathrm{mdim}}_M(Z_i, \{F_n\}_{n\in \mathbb{N}}, d)\). The same conclusion also holds for \(\overline{\mathrm{mdim}}^P_M\) and \(\overline{\mathrm{mdim}}^B_M\).\footnote{The same argument can not claim \(\underline{\mathrm{mdim}}_M(Z, \{F_n\}_{n\in \mathbb{N}}, d) = \max_{1\le i \le m} \underline{\mathrm{mdim}}_M(Z_i, \{F_n\}_{n\in \mathbb{N}}, d)\) (similarly for \(\underline{\mathrm{mdim}}^P_M\) and \(\underline{\mathrm{mdim}}^B_M\)). In fact, it may be trivial to have an example with $\liminf \max > \max \liminf$.}
			
			\item
			Let $\{F_{n}\}_{n\in\mathbb{N}}$ be a F\o lner sequence satisfying \eqref{1.1}, and let $\delta>0$. Then
			\begin{equation} \label{1:20}
				h^B_{\mathrm{top}}(Z, \delta, \{F_n\}_{n\in \mathbb{N}}, d) \le h^P_{\mathrm{top}}(Z, \frac{\delta}{2}, \{F_n\}_{n\in \mathbb{N}}, d) \le h_{\mathrm{top}}(Z, \frac{\delta}{2}, \{F_n\}_{n\in \mathbb{N}}, d),
			\end{equation}
			and then
			\[
			\overline{\mathrm{mdim}}^B_M(Z, \{F_n\}_{n\in \mathbb{N}}, d) \le \overline{\mathrm{mdim}}^P_M(Z, \{F_n\}_{n\in \mathbb{N}}, d) \le \overline{\mathrm{mdim}}_M(Z, \{F_n\}_{n\in \mathbb{N}}, d).
			\]
			Consequently, for each $x\in X$ one has
			\[\overline{\mathrm{mdim}}^B_M(x, \{F_n\}_{n\in \mathbb{N}}, d) \le \overline{\mathrm{mdim}}^P_M(x, \{F_n\}_{n\in \mathbb{N}}, d) \le \overline{\mathrm{mdim}}_M(x, \{F_n\}_{n\in \mathbb{N}}, d).
			\]
		\end{enumerate}
	\end{prop}
	
	\begin{proof}
		(1) comes directly from the definition.
		
		(2) It is direct to obtain
		\(h_{\mathrm{top}}(Z, \delta, \{F_n\}_{n\in \mathbb{N}}, d) = \max_{1\le i \le m} h_{\mathrm{top}}(Z_i, \delta, \{F_n\}_{n\in \mathbb{N}}, d)\)
		from the definition and then
		the required identity for \(\overline{\mathrm{mdim}}_M\).
		
		The identity \(h_{\mathrm{top}}^P (Z, \delta, \{F_n\}_{n\in \mathbb{N}}, d) = \max_{1\le i \le m} h_{\mathrm{top}}^P (Z_i, \delta, \{F_n\}_{n\in \mathbb{N}}, d)\) was proved as \cite[Proposition 2.1 (2)]{Dou-Zheng-Zhou2023}, and then  the required identity for \(\overline{\mathrm{mdim}}_M^P\) follows readily.

		Now let us prove the required identity for \(\overline{\mathrm{mdim}}_M^B\).
		It suffices to prove the inequality
		\begin{equation}
			\label{12}
			\overline{\mathrm{mdim}}_{M}^{B}(Z,\{F_{n}\}_{n\in\mathbb{N}},d)\leq  \max_{1\leq i\leq m}\overline{\mathrm{mdim}}_{M}^{B}(Z_{i},\{F_{n}\}_{n\in\mathbb{N}},d).
		\end{equation}
		By definition, it is not hard to notice that, for arbitrarily given $\delta>0$ and $\lambda>0$, one has
		\[
		F_{\lambda}(\delta,Z,\{F_{n}\}_{n\in\mathbb{N}},d)\leq \sum_{i=1}^{m}F_{\lambda}(\delta,Z_{i},\{F_{n}\}_{n\in\mathbb{N}},d)
		\]
		and then
		\[
		h_{\text{top}}^{B}(Z,\delta,\{F_{n}\}_{n\in\mathbb{N}},d)\leq \max_{1\leq i\leq m} h_{\text{top}}^{B}(Z_{i},\delta,\{F_{n}\}_{n\in\mathbb{N}},d),
		\]
		from which the inequality \eqref{12} follows readily.
		
		(3) It suffices to prove the inequalities \eqref{1:20}, which were essentially proved in \cite[Propositions 2.2 and 2.4]{Dou-Zheng-Zhou2023}.
		Let us fix arbitrarily given $\delta> 0$.
		
		Let $\varepsilon > 0$ and $\lambda>
		0$. Fix any $F_{m}$ with $|F_{m}|\geq \frac{1}{\varepsilon}$ (note that the group $G$ has assumed to be a countably
		infinite discrete amenable group, $|F_{m}|\geq \frac{1}{\varepsilon}$ once $m\in \mathbb{N}$ is large enough). Let $H\subset Z$ be an $(F_{m},\delta)$-separated set (with maximal cardinality). Then, $H$ is also an $(F_{m},\delta)$-spanning set for $Z$ and
		\(\{B^{F_{m}}_{\frac{\delta}{2}} (x)\}_{x\in H}\) are pairwise disjoint, and so
		\[
		F_{\lambda,\varepsilon}(\delta,Z,\{F_{n}\}_{n\in\mathbb{N}},d)\leq |H|\exp(-\lambda|F_{m}|) \leq P_{\lambda,\varepsilon}(\frac{\delta}{2},Z,\{F_{n}\}_{n\in\mathbb{N}},d).
		\]
		Thus for any $\{Z_{i}:i\in\mathbb{N}\}$ with $Z\subset \bigcup_{i=1}^{\infty}Z_{i}$ we have
		\[
		F_{\lambda,\varepsilon}(\delta,Z,\{F_{n}\}_{n\in\mathbb{N}},d)\leq \sum_{i=1}^{\infty} F_{\lambda,\varepsilon}(\delta,Z_{i},\{F_{n}\}_{n\in\mathbb{N}},d)\leq \sum_{i=1}^{\infty} P_{\lambda,\varepsilon}(\frac{\delta}{2},Z_{i},\{F_{n}\}_{n\in\mathbb{N}},d),
		\]
		which implies
		\[
		F_{\lambda}(\delta,Z,\{F_{n}\}_{n\in\mathbb{N}},d)\leq \widetilde{P}_{\lambda}(\frac{\delta}{2},Z,\{F_{n}\}_{n\in\mathbb{N}},d)
		\]
		and hence
		$h_{\text{top}}^{B}(Z,\delta,\{F_{n}\}_{n\in\mathbb{N}},d)\leq h_{\text{top}}^{P}(Z,\frac{\delta}{2},\{F_{n}\}_{n\in\mathbb{N}},d)$.	
		
		\smallskip
		
		Now let us prove the remainder inequality in \eqref{1:20}.
		Without loss of generality we may assume $h_{\text{top}}^{P}(Z,\frac{\delta}{2},\{F_{n}\}_{n\in\mathbb{N}},d) > 0$, and take $0< t<h_{\text{top}}^{P}(Z,\frac{\delta}{2},\{F_{n}\}_{n\in\mathbb{N}},d)$. Then
		\[P_{t}(\frac{\delta}{2},Z,\{F_{n}\}_{n\in\mathbb{N}},d)\ge \widetilde{P}_{t}(\frac{\delta}{2},Z,\{F_{n}\}_{n\in\mathbb{N}},d)=\infty.\]
		Thus, for each $\varepsilon> 0$, there exists a disjoint countable family $\{B_{\frac{\delta}{2}}^{F_{n_{i}}}(x_{i})\}_{i\in I}$, with $x_{i}\in Z$ and $|F_{n_{i}}|\geq \frac{1}{\varepsilon}$ for each $i\in I$, such that
		$\sum_{i\in I}\exp(-t|F_{n_{i}}|)>\sum_{m\geq 1}\frac{1}{m^{2}}$, and then
		\[
		\sum_{i\in I}\exp(-t|F_{n_{i}}|)=\sum_{m=1}^{\infty}\sum_{\begin{subarray}{c}
				i\in I
				\\F_{n_{i}}=F_{m}
		\end{subarray}}\exp(-t|F_{m}|)>\sum_{m=1}^{\infty}\frac{1}{m^{2}};
		\]
		in particular, there exists some integer $m = m_\varepsilon \in \mathbb{N}$ such that
		\[
		|F_{m}|\geq \frac{1}{\varepsilon}\ \text{and}\ \sum_{\begin{subarray}{c}
				i\in I
				\\F_{n_{i}}=F_{m}
		\end{subarray}}\exp(-t|F_{m}|)>\frac{1}{m^{2}},
		\]
		and so, by noting that $\{x_{i}:i\in I,\ F_{n_{i}}=F_{m}\}\subset Z$ is $(F_{m},\frac{\delta}{2})$-separated set, one has
		\[s_{F_{m}}(Z,\frac{\delta}{2},d)\geq \frac{1}{m^{2}}\cdot \exp(t|F_{m}|).\]
		We remark that, when $\varepsilon$ tends to $0$, $|F_m|$ tends to $\infty$ and then $m$ tends to $\infty$. This implies
		\begin{eqnarray*}
			h_{\text{top}}(Z,\frac{\delta}{2},\{F_{n}\}_{n\in\mathbb{N}},d) & = & \limsup_{n\to\infty} \frac{1}{|F_n|} \log s_{F_n}(Z, \frac{\delta}{2}, d) \\
			& \ge & \limsup_{\varepsilon\rightarrow 0} \frac{1}{|F_{m_\varepsilon}|} \log s_{F_{m_\varepsilon}}(Z, \frac{\delta}{2}, d) \\
			& \ge & \limsup_{\varepsilon\rightarrow 0} \frac{t |F_{m_\varepsilon}| - 2 \log m_\varepsilon}{|F_{m_\varepsilon}|} = t\ (\text{using \eqref{1.1}}).
		\end{eqnarray*}
		Then the required inequality
		\(
		h_{\text{top}}(Z,\frac{\delta}{2},\{F_{n}\}_{n\in\mathbb{N}},d)\geq h_{\text{top}}^{P}(Z,\frac{\delta}{2},\{F_{n}\}_{n\in\mathbb{N}})
		\)
		comes from the arbitrariness of $0< t<h_{\text{top}}^{P}(Z,\frac{\delta}{2},\{F_{n}\}_{n\in\mathbb{N}},d)$, finishing the proof of (3).
	\end{proof}

	\section{Proofs of Theorems \ref{thm5} and \ref{thm6}} \label{main}
	
	This section is devoted to the proofs of Theorems \ref{thm5} and \ref{thm6}.
	
	We begin with a key technical proposition that provides a uniform estimate on covering numbers. Its statement is inspired by \cite[Proposition 2.2]{R.Bowen1972-TAMS} and \cite[Proposition 4.3]{Dou-Wang-Zhang2025}, while its proof is also inspired by the proof of \cite[Proposition 2.6]{Dou-Wang-Zhang2025}.
	
	\begin{prop}\label{prop14}
		Let $\{F_{n}\}_{n\in\mathbb{N}}$ be a tempered F\o lner sequence
			with constant $C$, satisfying the growth condition \eqref{1.1} and $e_G\in F_n$ for every $n\in\mathbb N$. Let
			\(\mathcal{U} \in \mathfrak{C}_X^o\) and \(\eta > 0\). Assume that $\lambda\in \mathbb{R}$ satisfies
		\[
		\sup_{x \in X} h_{\mathrm{dim}}(\Gamma_\eta(x), \mathcal{U}, \{F_n\}_{n\in \mathbb{N}}) < \lambda,
		\]
		then, for each sufficiently small \(\delta > 0\), there exists \(P \in \mathbb{N}\) such that, for all \(n > P\),
		\[
		\sup_{x \in X} N(\mathcal{U}^{F_n}, B_\eta^{F_n}(x)) \le |\mathcal{U}|^{(\delta + \delta^{\frac{1}{4}})|F_n|} \mathrm{e}^{\lambda |F_n|} 2^{\delta |F_n|}.
		\]
	\end{prop}
	
	\begin{proof}
	Pick $\lambda^{*}\in \mathbb{R}$ and let $\delta\in (0,20^{-4})$ be small enough such that
			\begin{equation} \label{cstr}
				\lambda_{\delta}:=(1-10\delta^{\frac{1}{4}})\lambda>\lambda^{*}>\sup_{x\in X}h_{\mathrm{dim}}(\Gamma_{\eta}(x),\mathcal{U},\{F_{n}\}_{n\in\mathbb{N}}).
		\end{equation}	
		
		We take $M\in \mathbb{N}$ large enough (depending on $\delta$, $C$ and $D=\{e_{G}\}$) by Lemma \ref{coveringlem}, we also take $N\in \mathbb{N}$ by Lemma \ref{combilem} according to $\delta$. Since \(G\) is infinite and \(\{F_n\}_{n\in\mathbb N}\) is a F\o lner sequence,
			we have \(|F_n|\to\infty\). Hence, there exists $N'\in\mathbb{N}$ such that for every \(n\ge N'\),
			\begin{equation}\label{expo}
				|F_n|\leq \mathrm{e}^{(\lambda_{\delta}-\lambda^{*})|F_n|}.
			\end{equation}
			We shall denote from now on $\widetilde{N}=\max\{N,N'\}$ for simplicity.		
		
		\smallskip
		
		By the construction \eqref{cstr}, for $\varepsilon_1= (|F_{1}|+\cdots+|F_{\widetilde{N}}|)^{- 1} > 0$ and each $x \in X$, one has
		\begin{equation}\label{epsilon}
			0=F_{\lambda^{*}}(\mathcal{U},\Gamma_{\eta}(x),\{F_{n}\}_{n\in\mathbb{N}})\geq F_{\lambda^{*},\varepsilon_1}(\mathcal{U},\Gamma_{\eta}(x),\{F_{n}\}_{n\in\mathbb{N}}),
		\end{equation}
		and then, by the compactness of $\Gamma_{\eta}(x)$, there exist
		\begin{itemize}
			
			\item finitely many integers ${(\widetilde{N}\leq)}\ R_{x,1}<\cdots<R_{x,N_{x}}$, and
			
			\item finitely many open subsets $\{B_{x,1,r}\}_{r=1}^{k(x,1)}\subset \mathcal{U}^{F_{R_{x,1}}},
			\cdots,\{B_{x,N_{x},r}\}_{r=1}^{k(x,N_{x})}\subset \mathcal{U}^{F_{R_{x,N_{x}}}}$,
		\end{itemize}
		such that
		\[\Gamma_{\eta}(x)\subset U(x):=\bigcup_{j=1}^{N_{x}}\bigcup_{r=1}^{k(x,j)}B_{x,j,r}\ \ \text{and}\ \ \sum_{j=1}^{N_{x}}k(x,j)\mathrm{e}^{-\lambda^{*}|F_{R_{x,j}}|}<1.\]
		After merging possible repetitions among the sets \(F_{R_{x,j}}\), we may assume
			$
			F_{R_{x,1}},\cdots,F_{R_{x,N_x}}
			$
			are pairwise distinct. Furthermore, there exists a finite subset $S(x)\subset G$ and an open neighborhood of $x$, say $V(x)$, such that for each $x^{*}\in V(x)$, $B_{\eta}^{S(x)}(x^{*})\subset U(x)$.
		
		\smallskip
		
		To apply Lindenstrauss covering lemma (Lemma \ref{coveringlem}), let us construct $H_{i,j}$ as follows.
		
		\smallskip
		
		For $i=1$, note that we have obtained an open cover $\{V(x):x\in X\}$ as above. By the compactness of $X$, there exists $\{x_{1,1},\cdots,x_{1,P_{1}}\}\subset X$ such that $X= \bigcup_{q=1}^{P_{1}}V(x_{1,q})$. Let us rename $V(x_{1,q})$ to $V_{1,q}$ for each $q = 1, \cdots, P_1$ for simplicity. By rewriting those finitely many integers and finitely many open subsets for all points $x_{1,1},\cdots,x_{1,P_{1}}$,
		one has that:
		\begin{itemize}
			
			\item There exist finitely many integers $
			{(\widetilde{N}\leq)}\ R_{1,1}<\cdots<R_{1,N_{1}}
			$ with  $F_{R_{1,1}},\cdots,F_{R_{1,N_1}}$ pairwise distinct.
			
			\item For each $q\in\{1,\cdots,P_{1}\}$, there exist finitely many open subsets
			$$
			\{B_{1,q,r}^{(1)}\}_{r=1}^{k_{1}(1,q)}\subset \mathcal{U}^{F_{R_{1,1}}},\cdots,\{B_{N_{1},q,r}^{(1)}\}_{r=1}^{k_{1}(N_{1},q)}\subset \mathcal{U}^{F_{R_{1,N_{1}}}},$$
			some of $k_1 (j, q)$ for $j=1, \cdots, N_{1}$ may be zero, such that
			\[\Gamma_{\eta}(x_{1,q})\subset U_{1,q}:=\bigcup_{j=1}^{N_{1}}\bigcup_{r=1}^{k_{1}(j,q)}B_{j,q,r}^{(1)}\ \ \text{and}\ \ \sum_{j=1}^{N_{1}}k_{1}(j,q)\mathrm{e}^{-\lambda^{*}|F_{R_{1,j}}|}<1,\]
			additionally, there exists $S_{1,q}\in \mathfrak{F}_G$
			such that
			$B_{\eta}^{S_{1,q}}(x^{*})\subset U_{1,q}$
			for each $x^{*}\in V_{1,q}$.
			
			\item Set $S_{1}=\bigcup_{q=1}^{P_{1}}S_{1,q}$, which is a finite subset of $G$. Note that it holds
			$B_{\eta}^{S_{1}}(x^{*})\subset B_{\eta}^{S_{1,q}}(x^{*})\subset U_{1,q}$ for each $q\in\{1,\cdots,P_{1}\}$ and any $x^{*}\in V_{1,q}$.
		\end{itemize}
		
		\smallskip

		For each $i=2,\cdots,M$, as was done in the previous paragraphs (the only difference being that we shall consider $F_{\lambda^{*},\varepsilon_i}$ for $\varepsilon_i$ instead of $\varepsilon_1$, where $\varepsilon_i > 0$ will be specified later),
		there exists a finite subset
		\(\{x_{i,1},\cdots, x_{i,P_{i}}\}\subset X\), and a finite open cover \(\{ V_{i,1},\cdots,V_{i,P_{i}}\}\) of $X$ with $x_{i,q}\in V_{i,q}$ for each $q=1,\cdots,P_{i}$, furthermore:
		\begin{itemize}
			
			\item There exist finitely many integers
			$R_{i,1}<\cdots<R_{i,N_{i}}$ such that each $|F_{R_{i,j}}|\ge \frac{1}{\varepsilon_i}$, and the sets $F_{R_{i,1}},\cdots,F_{R_i,N_i}$ are pairwise distinct.
			
			\item For each $q\in\{1,\cdots,P_{i}\}$, there exist finitely many open subsets
			\[\{B_{1,q,r}^{(i)}\}_{r=1}^{k_{i}(1,q)}\subset \mathcal{U}^{F_{R_{i,1}}},\cdots,\{B_{N_{i},q,r}^{(i)}\}_{r=1}^{k_{i}(N_{i},q)}\subset \mathcal{U}^{F_{R_{i,N_{i}}}},
			\]
			some of $k_i (j, q)$ for $j=1, \cdots, N_{i}$ may be zero, such that
			\[\Gamma_{\eta}(x_{i,q})\subset U_{i,q}:=\bigcup_{j=1}^{N_{i}}\bigcup_{r=1}^{k_{i}(j,q)}B_{j,q,r}^{(i)}\ \ \text{and}\ \ \sum_{j=1}^{N_{i}}k_{i}(j,q)\mathrm{e}^{-\lambda^{*}|F_{R_{i,j}}|}<1.\]
			
			\item
			There exists $S_i\in \mathfrak{F}_G$ with $B_{\eta}^{S_{i}}(x^{*})\subset U_{i,q}$ for each
			$q\in\{1,\cdots,P_{i}\}$ and any $x^{*}\in V_{i,q}$.
		\end{itemize}
		Note that $G$ is assumed to be a countably infinite discrete amenable group and \(\{F_n\}_{n\in \mathbb{N}}\) is a F\o lner sequence. By choosing
		$\varepsilon_i> 0$ small enough, we may assume additionally that
		\begin{equation} \label{strong}
			\max_{p= 1}^{N_{i-1}} |F_{R_{i-1,p}}| < \min_{q= 1}^{N_i} |F_{R_{i,q}}|\ \ \text{and}\ \ \left|\bigcup_{i'<i}\bigcup_{j=1}^{N_{i'}}\{e_{G}\}F_{R_{i',j}}^{-1}F_{R_{i,k}}\right|\leq (1+\delta)|F_{R_{i,k}}|
		\end{equation}
		for each $i=2,\cdots,M$ and every $k= 1,\cdots,N_{i}$.
		
		The required $H_{i,j}$ is defined as $F_{R_{i,j}}$ for each $i=1,\cdots,M$ and every $j= 1,\cdots,N_{i}$.
		
		\smallskip
		
		Since \(\{F_n\}_{n \in \mathbb{N}}\) is a F\o lner sequence,
		there exists \(Q_1 \in \mathbb{N}\) (depending on the previously constructed \(H_{i,j}\), \(S_i\), and \(\delta\)) such that for all \(n \ge Q_1\), we have \(|F_n^*| \ge (1-\delta)|F_n|\), where
		\[
		F_n^* = \bigl\{ f \in F_n : H_{i,j}f \subset F_n, \; S_i f \subset F_n \text{ for each } i=1,\cdots,M \text{ and every } j=1,\cdots,N_i \bigr\}.
		\]
		
		Fix any sufficiently large \(n > Q_1\) and any \(x \in X\); we now estimate \(N\bigl(\mathcal{U}^{F_n}, B_\eta^{F_n}(x)\bigr)\).
		
		For each $i\in\{1,\cdots,M\}$ and every $g \in F_{n}^{*}$, notice that $\{V_{i,j}\}_{j=1}^{P_{i}}$ is an open cover of $X$, then $gx\in V_{i,i(g)}$ for some $i(g)\in \{1,\cdots,P_{i}\}$, and thus
		\[
		B_{\eta}^{S_{i}}(gx)\subset U_{i,i(g)}=\bigcup_{j=1}^{N_{i}}\bigcup_{r=1}^{k_{i}(j,i(g))}B_{j,i(g),r}^{(i)}.
		\]

		Pick arbitrarily $y\in B_{\eta}^{F_{n}}(x)$. Note
		\(
		gy\in gB_{\eta}^{F_{n}}(x)\subset gB_{\eta}^{S_{i}g}(x)=B_{\eta}^{S_{i}}(gx)
		\) (recall $g \in F_{n}^{*}$, so $S_i g\subset F_n$). For each $j= 1, \cdots, N_i$ we set
		\[
		A_{i,j}(y)=\{g\in F_{n}^{*}:gy\in \bigcup_{r=1}^{k_{i}(j,i(g))}B_{j,i(g),r}^{(i)}\}.
		\]
		It should be noted that \( i(g) \) depends only on \( i \), \( g \), and \( x \), and is independent of \( y \).
		Then
		\[
		F_{n}^{*}=\bigcup_{j=1}^{N_{i}}A_{i,j}(y) \ \ \text{and}\ \  H_{i,j}A_{i,j}(y)\subset F_{n}.
		\]
		Consider the collection $\{H_{i,j}a:a\in A_{i,j}(y);i=1,\cdots,M;\, j=1,\cdots,N_{i}\}$. It admits a $10\delta^\frac{1}{4}$-disjoint sub-collection $\mathcal{F}_{y}$ satisfying $|\cup \mathcal{F}_{y}|\geq (1-\delta-\delta^{\frac{1}{4}})|F_{n}|$ by Lemma \ref{coveringlem}.
		
		\smallskip
		
		Let $\mathfrak{F}=\{\mathcal{F}_{y}:y\in B_{\eta}^{F_{n}}(x)\}$. Note that each family $\mathcal{F}_{y}$ is a $10\delta^{\frac{1}{4}}$-disjoint, and hence $\frac{1}{2}$-disjoint because $\delta\in (0,20^{-4})$, sub-collection of $\{F_{R_{i,j}}a : F_{R_{i,j}}a\subset F_{n};\,i=1,\cdots,M;\, j= 1,\cdots,N_{i}\}$ (recall $H_{i, j} = F_{R_{i,j}}$), and that $N\in \mathbb{N}$ is constructed by Lemma \ref{combilem} according to $\delta$. Using  Lemma \ref{combilem} one has that, there exists $Q_{2}\in \mathbb{N}$ (depending on the previously constructed $R_{i,j}$) such that $|\mathfrak{F}|\leq 2^{\delta|F_{n}|}$ for every $n\ge Q_{2}$.
		
		\smallskip
		
		Fix arbitrarily given $\mathcal{F}\in \mathfrak{F}$, and define $W_{\mathcal{F}} = \{y\in B_{\eta}^{F_{n}}(x) : \mathcal{F}_{y}=\mathcal{F}\}$.
		Now let us show
		\begin{equation} \label{claim}
			N(\mathcal{U}^{F_{n}},W_{\mathcal{F}})\leq \mathrm{e}^{\lambda|F_{n}|}|\mathcal{U}|^{(\delta+\delta^{\frac{1}{4}})|F_{n}|}.
		\end{equation}

Take arbitrarily \(F\in \mathcal{F}\). Then \(F\) is exactly
			\(F_{R_{i,j}}a\) for some \(i\in \{1,\cdots,M\}\),
			\(j\in\{1,\cdots,N_i\}\), and \(a\in F_n^*\). In fact, such \(i\)
			exists uniquely by the construction \eqref{strong}; denote it by \(i_F\).
			Set
			$
			\mathcal{E}_F
			=
			\{
			a\in F_n^*:
			\text{ there exists }j\in\{1,\cdots,N_{i_F}\}
			\text{ such that }F=F_{R_{i_F,j}}a
			\}.
			$
			Recall that for each \(i\), the sets
			\(F_{R_{i,1}},\cdots,F_{R_{i,N_i}}\) are pairwise distinct. Thus for every \(a\in \mathcal{E}_F\),
there exists \(j_F(a)\in\{1,\cdots,N_{i_F}\}\) uniquely
			such that
			$
			F=F_{R_{i_F,j_F(a)}}a.
			$
			We now claim
			\begin{equation}\label{estimate WF}
				W_{\mathcal F}
				\subset
				\bigcup_{a\in \mathcal{E}_F}
				\bigcup_{r=1}^{k_{i_F}(j_F(a),i_F(a))}
				a^{-1}B_{j_F(a),i_F(a),r}^{(i_F)} .
			\end{equation}
			Indeed, take arbitrarily \(y\in W_{\mathcal F}\). Since \(\mathcal F_y=\mathcal F\)
			and \(F\in\mathcal F\), by the construction of \(\mathcal F_y\) there
			exist \(j\in\{1,\cdots,N_{i_F}\}\) and \(a\in A_{i_F,j}(y)\) such that
			$
			F=F_{R_{i_F,j}}a,
			$
			and hence \(a\in \mathcal{E}_F\) and \(j=j_F(a)\). Then \eqref{estimate WF} follows from the definition of
			\(A_{i_F,j_F(a)}(y)\), as
			\[
			ay\in
			\bigcup_{r=1}^{k_{i_F}(j_F(a),i_F(a))}
			B_{j_F(a),i_F(a),r}^{(i_F)},\quad\text{equivalently,}\quad y\in
			\bigcup_{r=1}^{k_{i_F}(j_F(a),i_F(a))}
			a^{-1}B_{j_F(a),i_F(a),r}^{(i_F)}.
			\]
Note that the sets appearing in the right-hand side of (\ref{estimate WF})
			are elements of \(\mathcal U^F\) since \(F=F_{R_{i_F,j_F(a)}}a\).
			We also note that \(\mathcal E_F\) has cardinality at most \(|F|\). In fact,
			since \(e_G\in F_{R_{i_F,j_F(a)}}\), every \(a\in\mathcal E_F\) satisfies
			$
			a\in F_{R_{i_F,j_F(a)}}a=F.
			$
			Thus \(\mathcal E_F\subset F\), and hence
			$
			|\mathcal E_F|\le |F|.
			$
			Therefore,
			\begin{equation}\label{dk}
				\begin{aligned}
					N(\mathcal U^F,W_{\mathcal F})
					&\le
					\sum_{a\in \mathcal{E}_F}
					k_{i_F}(j_F(a),i_F(a))\ \ \ (\mbox{by}\ (\ref{estimate WF}))
					\\&= \left(\sum_{a\in \mathcal{E}_F}k_{i_F}(j_F(a),i_F(a))\mathrm{e}^{-\lambda^{*}|F_{R_{i_F,j_F(a)}}|}\right)\cdot \mathrm{e}^{\lambda^{*}|F|} \\&<|\mathcal{E}_{F}|\cdot  \mathrm{e}^{\lambda^{*}|F|}\ \ \ (\mbox{as}\,\, k_{i_F}(j_F(a),i_F(a))\mathrm{e}^{-\lambda^{*}|F_{R_{i_F,j_F(a)}}|}<1)
					\\&\leq |F|\cdot \mathrm{e}^{(\lambda^{*}-\lambda_{\delta})|F|}\cdot \mathrm{e}^{\lambda_{\delta}|F|}\leq \mathrm{e}^{\lambda_{\delta}|F|}\ \ \ (\text{by}\ (\ref{expo})\  \text{and}\ R_{i_F,j_F(a)}\geq \widetilde{N}\geq N').
				\end{aligned}
		\end{equation}
		
		As $|F_{n}\setminus\cup \mathcal{F}|\leq (\delta + \delta^{\frac{1}{4}})|F_{n}|$ by the construction of applying Lemma \ref{coveringlem}, one has

		\[
		\begin{aligned}
			N(\mathcal{U}^{F_{n}},W_{\mathcal{F}})&\leq |\mathcal{U}|^{|F_{n}\setminus\cup \mathcal{F}|} \cdot\prod_{F\in \mathcal{F}}
			N (\mathcal{U}^{F}, W_{\mathcal{F}})
			\\&\leq |\mathcal{U}|^{(\delta+\delta^{\frac{1}{4}})|F_{n}|}\cdot\prod_{F\in \mathcal{F}}\mathrm{e}^{{{\lambda_{\delta}}}|F|}\ \ \ (\text{using \eqref{dk}})
			\\&\leq |\mathcal{U}|^{(\delta+\delta^{\frac{1}{4}})|F_{n}|}\cdot\exp{\Bigg({\frac{{{\lambda_{\delta}}}|\cup \mathcal{F}|}{1-10\delta^{\frac{1}{4}}}}\Bigg)}\ \ \ (\mbox{as}\,\, \mathcal{F} \,\,\text{is}\,\,10\delta^{\frac{1}{4}} \text{-disjoint})
			\\&\leq  |\mathcal{U}|^{(\delta+\delta^{\frac{1}{4}})|F_{n}|}\cdot \mathrm{e}^{\lambda|F_{n}|}\ \ \ (\text{note}\ {{\lambda_{\delta}}}=(1-10\delta^{\frac{1}{4}})\lambda\ \text{and}\ \cup \mathcal{F}\subset F_n).
		\end{aligned}
		\]
		This finishes the proof of \eqref{claim}.

		Finally, combining \eqref{claim} with the bound on \(|\mathfrak{F}|\), one has that
		\[
		N(\mathcal{U}^{F_n}, B_\eta^{F_n}(x)) \le \sum_{\mathcal{F} \in \mathfrak{F}} N(\mathcal{U}^{F_n}, W_\mathcal{F}) \le |\mathcal{U}|^{(\delta+\delta^{\frac{1}{4}})|F_n|} e^{\lambda |F_n|} 2^{\delta |F_n|}
		\]
		whenever $n>P=\max\{Q_{1},Q_{2}\}$, completing the proof.
	\end{proof}

	We shall also make use of the following technical yet straightforward fact.
	
	\begin{lemma}\label{lem21}
		For any $\delta > 0$, there exists $\mathcal{U} \in \mathfrak{C}_{X}^{o}$ such that its Lebesgue number satisfies \(\mathcal{L}(\mathcal{U}) \ge \delta\) (that is, for any given subset $A\subset X$, once its diameter satisfies $\mathrm{diam} (A)< \delta$ one has that $A$ is contained in some element of the cover $\mathcal{U}$), its diameter satisfies \(\mathrm{diam}(\mathcal{U},d) \le 3\delta\), and its cardinality is \(|\mathcal{U}| = r_{\{e_G\}}(X, \frac{\delta}{2}, d)\).
	\end{lemma}
	
	\begin{proof}
		Let $E$ be an $(\{e_{G}\},\frac{\delta}{2})$-spanning set of $X$ with $|E| = r_{\{e_{G}\}}(X,\frac{\delta}{2},d)$. Consider $\mathcal{U} = \left\{O_{\frac{3}{2}\delta}(x) : x \in E\right\}$. Then $\mathcal{U}$ is a finite open cover of $X$ with $\mathrm{diam}(\mathcal{U},d) \le 3\delta$, and $\mathcal{L}(\mathcal{U}) \ge \delta$ follows directly from the triangle inequality.
	\end{proof}
	
	We are now ready to prove the main theorems.
	
	\begin{proof}[Proof of Theorem \ref{thm5}]
		Let us prove firstly
		\begin{equation}\label{prop22}
			\overline{\mathrm{mdim}}_{M}^{B}(X,\{F_{n}\}_{n\in\mathbb{N}},d)= \overline{\mathrm{mdim}}_{M}^{P}(X,\{F_{n}\}_{n\in\mathbb{N}},d)=
			\overline{\mathrm{mdim}}_{M}(X,\{F_{n}\}_{n\in\mathbb{N}},d).
		\end{equation}
		By Proposition \ref{prop20} (3)
		it remains to show $\overline{\mathrm{mdim}}_{M}(X,\{F_{n}\}_{n\in\mathbb{N}},d)\leq \overline{\mathrm{mdim}}_{M}^{B}(X,\{F_{n}\}_{n\in\mathbb{N}},d)$. In fact, we would prove that for any $\delta>0$
		\begin{equation} \label{xy}
			h_{\text{top}}(X,3\delta,\{F_{n}\}_{n\in\mathbb{N}},d)\leq h_{\text{top}}^{B}(X,\frac{\delta}{3},\{F_{n}\}_{n\in\mathbb{N}},d).
		\end{equation}
		
		\smallskip
		
		Fix arbitrarily given $\delta>0$. Let $\mathcal{U}$ be the corresponding open cover given by Lemma \ref{lem21} according to $\delta$. As $\mathcal{U}$ is a finite open cover of $X$
		such that its Lebesgue number satisfies \(\mathcal{L}(\mathcal{U}) \ge \delta\) and its diameter satisfies \(\mathrm{diam}(\mathcal{U},d) \le 3\delta\), then it is trivial obtain from the definitions the following trivial estimates
		\[
		h_{\text{top}}(X,3\delta,\{F_{n}\}_{n\in\mathbb{N}},d)\leq 	h_{\text{top}}(X,\mathcal{U},\{F_{n}\}_{n\in\mathbb{N}}),
		\]
		\[F_{\lambda, \varepsilon}(\mathcal{U}, X, \{F_n\}_{n\in \mathbb{N}})\le F_{\lambda, \varepsilon}(\frac{\delta}{3}, X, \{F_n\}_{n\in \mathbb{N}}, d)\] for all ${\lambda> 0}$ and $\varepsilon > 0$,
		and then
		\(
		h_{\mathrm{dim}}(X,\mathcal{U},\{F_{n}\}_{n\in\mathbb{N}})\leq h_{\text{top}}^{B}(X,\frac{\delta}{3},\{F_{n}\}_{n\in\mathbb{N}},d).
		\)
		We remark that, by letting $\mathcal{W} = \{X\}$ in \cite[Proposition 2.6]{Dou-Wang-Zhang2025}, one deduces directly
		\begin{equation*}\label{6.2}
			h_{\text{top}}(X,\mathcal{U},\{F_{n}\}_{n\in\mathbb{N}})= h_{\mathrm{dim}}(X,\mathcal{U},\{F_{n}\}_{n\in\mathbb{N}})
		\end{equation*}
		under the assumption that $\{F_{n}\}_{n\in\mathbb{N}}$ is a tempered F\o lner sequence satisfying \eqref{1.1}. Along with the above estimates, this establishes the required inequality \eqref{xy}.
		
		Trivially,
		\(
		\max_{x\in X}\overline{\mathrm{mdim}}_{M}(x,\{F_{n}\}_{n\in\mathbb{N}},d)\leq \overline{\mathrm{mdim}}_{M}(X,\{F_{n}\}_{n\in\mathbb{N}},d).
		\)
		In view of \eqref{prop22} and Proposition \ref{prop20} (3), to complete the proof, it suffices to show
		\begin{equation} \label{lem24}
			\overline{\mathrm{mdim}}_{M}^{B}(X,\{F_{n}\}_{n\in\mathbb{N}},d)\le \max_{x\in X}\overline{\mathrm{mdim}}_{M}^{B}(x,\{F_{n}\}_{n\in\mathbb{N}},d).
		\end{equation}
		
		Let $\mathcal{W}_{1}$ be a finite closed cover of $X$ with $\mathrm{diam}(\mathcal{W}_{1},d)\leq 1$. Then by Proposition \ref{prop20} (2), there exists $W_{1}\in \mathcal{W}_{1}$ such that  $\overline{\mathrm{mdim}}_{M}^{B}(X,\{F_{n}\}_{n\in\mathbb{N}},d)=\overline{\mathrm{mdim}}_{M}^{B}(W_{1},\{F_{n}\}_{n\in\mathbb{N}},d)$. Similarly, let $\mathcal{W}_2$ be a finite closed cover of $W_{1}$ with $\mathrm{diam}(\mathcal{W}_{2},d)\leq \frac{1}{2}$. By same argument, there exists $W_{2}\in \mathcal{W}_{2}$ such that  $\overline{\mathrm{mdim}}_{M}^{B}(X,\{F_{n}\}_{n\in\mathbb{N}},d)=\overline{\mathrm{mdim}}_{M}^{B}(W_{2},\{F_{n}\}_{n\in\mathbb{N}},d)$. Continuing this process, we obtain a sequence of closed sets $\{W_{m}\}_{m\in\mathbb{N}}$ such that $\mathrm{diam}({W_{m}},d)\leq \frac{1}{2^{m}}$, $W_{m+1}\subset W_{m}$ and
		\(
		\overline{\mathrm{mdim}}_{M}^{B}(X,\{F_{n}\}_{n\in\mathbb{N}},d)=\overline{\mathrm{mdim}}_{M}^{B}(W_{m},\{F_{n}\}_{n\in\mathbb{N}},d)
		\)
		for every $m\in \mathbb{N}$.
		By the Cantor intersection theorem, there is uniquely a point $y$ such that $\{y\}=  \bigcap_{m=1}^{\infty}W_{m}$. Thus along with Proposition \ref{prop20} (1) we obtain
		\[		\overline{\mathrm{mdim}}_{M}^{B}(X,\{F_{n}\}_{n\in\mathbb{N}},d)=\overline{\mathrm{mdim}}_{M}^{B}(y,\{F_{n}\}_{n\in\mathbb{N}},d)\le\max_{x\in X}\overline{\mathrm{mdim}}_{M}^{B}(x,\{F_{n}\}_{n\in\mathbb{N}},d).
		\]
		This proves \eqref{lem24} and then completes the proof.
	\end{proof}

	\begin{proof}[{Proof of Theorem \ref{thm6}}]
		Fix arbitrarily given $\varepsilon>0$. In view of the trivial fact that
		\[
		\sup_{x\in X}h_{\text{top}}(\Gamma_{\varepsilon}(x),\delta,\{F_{n}\}_{n\in\mathbb{N}},d)\leq 	h_{\text{top}}(X,\delta,\{F_{n}\}_{n\in\mathbb{N}},d),
		\]
		clearly
		\[
		\limsup_{\delta\to 0}\sup_{x\in X}\frac{h_{\text{top}}(\Gamma_{\varepsilon}(x),\delta,\{F_{n}\}_{n\in\mathbb{N}},d)}{\log\frac{1}{\delta}}
		\leq \overline{\mathrm{mdim}}_{M}(X,\{F_n\}_{n\in\mathbb{N}},d).
		\]
		With the help of {Proposition \ref{rmk-independent} and Proposition \ref{prop20} (3)}, it remains to show
		\begin{equation}\label{equ6.3}
			\overline{\mathrm{mdim}}_{M}(X,\{\widetilde F_n\}_{n\in\mathbb{N}},d)\leq 	\limsup_{\delta\to 0}\sup_{x\in X}\frac{h_{\text{top}}^{B}(\Gamma_{\varepsilon}(x),\frac{2 \delta}{3},\{F_{n}\}_{n\in\mathbb{N}},d)}{\log\frac{1}{\delta}},
		\end{equation}
		{where $
			\widetilde F_n:=F_n\cup\{e_G\}$, $n\in\mathbb N.
			$
			It is easy to see that \(\{\widetilde F_n\}_{n\in\mathbb N}\) is also a tempered F\o lner
			sequence, and satisfies
			\eqref{1.1}.
			We shall apply Proposition~\ref{prop14} to this auxiliary sequence.}
		
		Fix any $\delta>0$ and choose an open cover \(\mathcal{U}\) by Lemma \ref{lem21} corresponding to \(2\delta\).
		Take any
		$$\lambda>\sup_{x\in X}h_{\text{top}}^{B}(\Gamma_{\varepsilon}(x),\frac{2 \delta}{3},\{F_{n}\}_{n\in\mathbb{N}},d).$$
		As \(\mathcal{L}(\mathcal{U}) \ge 2\delta\), for all $\kappa > 0$ one has \(F_{\lambda, \kappa}(\mathcal{U}, \Gamma_\varepsilon(x), \{F_n\}_{n\in \mathbb{N}})\le F_{\lambda, \kappa}(\frac{2 \delta}{3}, \Gamma_\varepsilon(x), \{F_n\}_{n\in \mathbb{N}}, d)\). Furthermore, by the definition of Bowen's dimensional entropy in \S3.2.3, it follows that
			$F_{\lambda, \kappa}(\mathcal{U}, \Gamma_\varepsilon(x), \{\widetilde F_n\}_{n\in \mathbb{N}})\leq |\mathcal{U}|\cdot F_{\lambda, \kappa}(\mathcal{U}, \Gamma_\varepsilon(x), \{F_n\}_{n\in \mathbb{N}})$. Thus
			\[\sup_{x\in X} h_{\mathrm{dim}}(\Gamma_\varepsilon(x), \mathcal{U}, \{\widetilde F_n\}_{n\in \mathbb{N}}) < \lambda.
			\]
			
			Apply Proposition \ref{prop14}
			with \(\eta = \varepsilon\). For all sufficiently small \(\xi > 0\) and large enough \(n\),
			\[
			\sup_{x\in X} N(\mathcal{U}^{\widetilde F_n}, B_\varepsilon^{\widetilde F_n}(x)) \le |\mathcal{U}|^{(\xi+\xi^{\frac{1}{4}})|\widetilde F_n|} e^{\lambda |\widetilde F_n|} 2^{\xi |\widetilde F_n|},
			\]
			which implies a uniform upper bound for \(r_{\widetilde F_n}(B_\varepsilon^{\widetilde F_n}(x), 6\delta, d)\)
			by noting \(\mathrm{diam}(\mathcal{U},d) \le 6\delta\) and \(|\mathcal{U}| = r_{\{e_G\}}(X, \delta, d)\), that is,
			\[
			\sup_{x\in X}r_{\widetilde F_{n}}(B_{\varepsilon}^{\widetilde F_{n}}(x),6\delta,d)\leq |\mathcal{U}|^{(\xi+\xi^{\frac{1}{4}})|\widetilde F_{n}|}\mathrm{e}^{\lambda|\widetilde F_{n}|}2^{\xi|\widetilde F_{n}|}=r_{\{e_{G}\}}(X,\delta,d)^{(\xi+
				\xi^{\frac{1}{4}})|\widetilde F_{n}|}\cdot \mathrm{e}^{\lambda|\widetilde F_{n}|}2^{\xi|\widetilde F_{n}|}.
			\]
			For each $n\in \mathbb{N}$ large enough as above, let \(E_n\) be an \((\widetilde F_n, \varepsilon)\)-spanning set for \(X\) with
			$|E_{n}|=r_{\widetilde F_{n}}(X,\varepsilon,d)$. Clearly,
			$\{B_{\varepsilon}^{\widetilde F_{n}}(y): y\in E_n\}$ forms a finite cover of $X$, and
			then
			\[
			\begin{aligned}
				s_{\widetilde F_{n}}(X,12\delta,d)&\leq r_{\widetilde F_{n}}(X,6\delta,d) \ \ (\text{noting}\ \eqref{se-sp})\\
				&\leq \sum_{y\in E_{n}}r_{\widetilde F_{n}}(B_{\varepsilon}^{\widetilde F_{n}}(y),6\delta,d)
				\\&\leq r_{\widetilde F_{n}}(X,\varepsilon,d)\cdot r_{\{e_{G}\}}(X,\delta,d)^{(\xi+\xi^{\frac{1}{4}})|\widetilde F_{n}|}\cdot \mathrm{e}^{\lambda|\widetilde F_{n}|}2^{\xi|\widetilde F_{n}|}.
			\end{aligned}
			\]
			
			By taking logarithms, dividing by \(|\widetilde F_n|\), taking the \(\limsup\) as \(n \to \infty\), then letting \(\xi \to 0\) and
			\(\lambda \searrow \sup_{x\in X} h_{\mathrm{top}}^{B}\bigl(\Gamma_{\varepsilon}(x),\frac{2 \delta}{3},\{F_{n}\}_{n\in\mathbb{N}},d\bigr)\), we obtain
			\begin{equation*} \label{fir}
				h_{\mathrm{top}}(X,12\delta,\{\widetilde F_{n}\}_{n\in\mathbb{N}},d)\le
				\limsup_{n\to\infty} \frac{1}{|\widetilde F_n|} \log r_{\widetilde F_n} (X, \varepsilon, d)
				+
				\sup_{x\in X}h_{\mathrm{top}}^{B}\bigl(\Gamma_{\varepsilon}(x),\frac{2 \delta}{3},\{ F_{n}\}_{n\in\mathbb{N}},d\bigr).
			\end{equation*}
			Thus the required inequality \eqref{equ6.3} follows readily by dividing by \(\log\bigl(\frac{1}{12\delta}\bigr)\) and taking the \(\limsup\) as \(\delta\to 0\). We also use
			the elementary fact that for every fixed \(\varepsilon>0\),
			\begin{equation} \label{sec}
				\limsup_{n\to\infty} \frac{1}{|\widetilde F_n|} \log r_{\widetilde F_n} (X, \varepsilon, d)< \infty.
		\end{equation}
		In fact, if let $\mathcal{V}$ be a finite open cover of $X$ such that its diameter satisfies \(\mathrm{diam}(\mathcal{V},d) \le \varepsilon\), then
		\(r_F(X, \varepsilon, d)\le N (\mathcal{V}^{F}, X)\le |\mathcal{V}|^{|F|}\) for each $F\in \mathfrak{F}_G$, which implies directly \eqref{sec}.
	\end{proof}
	
	The temperedness condition on the F\o lner sequence \(\{F_n\}_{n\in \mathbb{N}}\) is essential in Theorems \ref{thm5} and \ref{thm6}, as it enables the application of the Lindenstrauss covering lemma (Lemma \ref{coveringlem}). A natural question is whether similar results hold for arbitrary F\o lner sequences. Just as commented in Remark \ref{tem}, we can extend \eqref{f2}, \eqref{2.2-2}, \eqref{2.2-3} and \eqref{f1} to arbitrary F\o lner sequences by
	using the standard technique of passing to a tempered subsequence and noting that topological entropy is monotone with respect to taking subsequences of F\o lner sequences.
	
	Below let us present firstly an easy fact whose proof comes directly from the definition.
	
	\begin{lemma} \label{DWZ-prop}
		Let $\{F_{n}^{*}\}_{n\in\mathbb{N}}$ be a subsequence of the F\o lner sequence $\{F_{n}\}_{n\in\mathbb{N}}$. Then
		\[
		h_{\mathrm{top}}(K,\delta,\{F_{n}^{*}\}_{n\in\mathbb{N}},d)\leq h_{\mathrm{top}}(K,\delta,\{F_{n}\}_{n\in\mathbb{N}},d)\
		\ \text{for every}\ \emptyset\neq K\subset X\text{ and }\delta>0.
		\]
	\end{lemma}
	
	Now let us extend \eqref{f2}, \eqref{2.2-2}, \eqref{2.2-3} and \eqref{f1} as follows.
	
	\begin{theorem} \label{without-tem}
		Let $\{F_n\}_{n\in \mathbb{N}}\subset \mathfrak{F}_{G}$ be an arbitrary F\o lner sequence and $\varepsilon>0$. Then
the identities \eqref{f2}, \eqref{2.2-2} and \eqref{2.2-3} still remain valid, furthermore,
\[
\overline{\mathrm{mdim}}_M(X, \{F_n\}_{n\in \mathbb{N}}, d) = \limsup_{\delta \to 0} \sup_{x \in X} \frac{h_{\mathrm{top}}(\Gamma_\varepsilon(x), \delta, \{F_n\}_{n\in \mathbb{N}}, d)}{\log \frac{1}{\delta}}.
\]
	\end{theorem}
	
	\begin{proof}
		{Note that the inequality \eqref{lem24} in the proof of Theorem \ref{thm5}, relies on neither the temperedness condition nor the growth condition (\ref{1.1}), and hence holds for arbitrary F\o lner sequences, from which the identity \eqref{2.2-3} follows directly. The same reasoning along with Proposition \ref{prop20} (2) implies the identities \eqref{f2} and \eqref{2.2-2}.} Now it suffices to prove
		\begin{equation}\label{move-tempered}
			\limsup_{\delta \to 0} \sup_{x \in X} \frac{h_{\mathrm{top}}(\Gamma_\varepsilon(x), \delta, \{F_n\}_{n\in \mathbb{N}}, d)}{\log \frac{1}{\delta}}\geq \overline{\mathrm{mdim}}_M(X, \{F_n\}_{n\in \mathbb{N}}, d).
		\end{equation}
		Note that for the F\o lner sequence $\{F_n\}_{n\in \mathbb{N}}$ there always exists a tempered subsequence $\{F_{n}^{*}\}_{n\in \mathbb{N}}$ satisfying the growth condition (\ref{1.1}). Applying (\ref{f1}) to $\{F_{n}^{*}\}_{n\in \mathbb{N}}$ we obtain
		\begin{equation}\label{3.12}
			\limsup_{\delta \to 0} \sup_{x \in X} \frac{h_{\mathrm{top}}(\Gamma_\varepsilon(x), \delta, \{F_n^{*}\}_{n\in \mathbb{N}}, d)}{\log \frac{1}{\delta}}= \overline{\mathrm{mdim}}_M(X, \{F_n^{*}\}_{n\in \mathbb{N}}, d).
		\end{equation}
		Recall that the global metric mean dimension is independent of the choice of F\o lner sequence ({Proposition} \refeq{rmk-independent}), then (\ref{move-tempered}) follows directly from Lemma \ref{DWZ-prop} and above (\ref{3.12}).
	\end{proof}

	\section{The construction of Theorem \ref{count}}\label{example}
	
	{\it Throughout the whole section we shall work with $\mathbb{Z}$-actions along with the canonical F\o lner sequence \(F_n =
		\{0, 1, \cdots, n-1\}\) for each $n\in\mathbb N$.} For the reader's convenience, for each $\varepsilon>0$
	let us recall from Section \ref{state} the following notations,
	where we omit the sequence \(\{F_n\}_{n\in\mathbb N}\) from the notation for simplicity as it is fixed throughout the whole section:
	\begin{equation}\label{def}
		\begin{aligned}
			\overline{D}_{\mathrm{int}}(X,\varepsilon,d)&:=\limsup_{\delta\to 0}\sup_{x\in X}
			\frac{h_{\mathrm{top}}(\Gamma_{\varepsilon}^{\{F_n\}_{n\in\mathbb N}}(x),\delta,\{F_n\}_{n\in\mathbb N},d)}
			{\log \frac{1}{\delta}},\\
			\underline{D}_{\mathrm{int}}(X,\varepsilon,d)&:=\liminf_{\delta\to 0}\sup_{x\in X}
			\frac{h_{\mathrm{top}}(\Gamma_{\varepsilon}^{\{F_n\}_{n\in\mathbb N}}(x),\delta,\{F_n\}_{n\in\mathbb N},d)}
			{\log \frac{1}{\delta}},\\
			\overline{D}_{\mathrm{ext}}(X,\varepsilon,d)&:=\sup_{x\in X}\limsup_{\delta\to 0}
			\frac{h_{\mathrm{top}}(\Gamma_{\varepsilon}^{\{F_n\}_{n\in\mathbb N}}(x),\delta,\{F_n\}_{n\in\mathbb N},d)}
			{\log \frac{1}{\delta}},\\
			\underline{D}_{\mathrm{ext}}(X,\varepsilon,d)&:=\sup_{x\in X}\liminf_{\delta\to 0}
			\frac{h_{\mathrm{top}}(\Gamma_{\varepsilon}^{\{F_n\}_{n\in\mathbb N}}(x),\delta,\{F_n\}_{n\in\mathbb N},d)}
			{\log \frac{1}{\delta}}.
		\end{aligned}
	\end{equation}
	
	In this section we present a family of $\mathbb{Z}$-actions for which the following four quantities
	\[
	\overline{D}_{\mathrm{int}},
	\quad
	\underline{D}_{\mathrm{int}},
	\quad
	\underline{D}_{\mathrm{ext}},
	\quad
	\overline{D}_{\mathrm{ext}}
	\]
	can be completely separated.
	As the first step, we construct a shift space $(X,\sigma)$ satisfying
	\begin{equation}\label{first-step-separation}
		\overline{D}_{\mathrm{int}}(X,\varepsilon,d)=
		\underline{D}_{\mathrm{int}}(X,\varepsilon,d)
		> \overline{D}_{\mathrm{ext}}(X,\varepsilon,d) =
		\underline{D}_{\mathrm{ext}}(X,\varepsilon,d)
	\end{equation}
	for every $\varepsilon\in(0,1)$. This construction demonstrates that the supremum and limit superior of topological entropy, and hence packing topological entropy and Bowen's dimensional entropy in the localization formula, cannot
		be interchanged (by Theorem \ref{thm6} and Proposition \ref{prop20} (3)).
Thus the answer to Question (1) (posed in \cite[\S 6]{Yang-Chen-Zhou2024}) is negative.
 To achieve a stronger separation, we introduce a family of systems $Y_E$ whose upper and lower parameters can be prescribed independently. Finally, by taking the product of $X$ with a suitable member of $Y_E$, we combine these two constructions to obtain a system in which all of above four quantities are pairwise distinct.

	\begin{remark}\label{rem:other-questions}
		The same construction can also be used to address \cite[\S 6, Questions~(2) and~(3)]{Yang-Chen-Zhou2024}. For clarity, let us briefly explain it as follows.
		
		Let \(T: W\rightarrow W\) be a homeomorphism on a compact metric space \((W,d)\). Set
		\[
		\Phi_{\varepsilon}(w,\delta,d)
		:=
		\limsup_{n\to\infty}
		\frac{1}{|F_n|}
		\log s_{F_n}
		\bigl(
		T^{-n}\Gamma^{\{F_m\}_{m\in\mathbb{N}}}_{\varepsilon}(w),
		\delta,d
		\bigr),
		\qquad w\in W,
		\]
		and then define the four corresponding preimage quantities by
		\[
		\begin{aligned}
			\overline{D}_{\mathrm{int}}^{\mathrm{pre}}(W,\varepsilon,d)
			&:=
			\limsup_{\delta\to0}
			\sup_{w\in W}\frac{
				\Phi_{\varepsilon}(w,\delta,d)
			}
			{\log\frac1\delta},
			&
			\underline{D}_{\mathrm{int}}^{\mathrm{pre}}(W,\varepsilon,d)
			&:=
			\liminf_{\delta\to0}
			\sup_{w\in W}\frac{
				\Phi_{\varepsilon}(w,\delta,d)
			}
			{\log\frac1\delta},
			\\
			\overline{D}_{\mathrm{ext}}^{\mathrm{pre}}(W,\varepsilon,d)
			&:=
			\sup_{w\in W}
			\limsup_{\delta\to0}
			\frac{
				\Phi_{\varepsilon}(w,\delta,d)
			}
			{\log\frac1\delta},
			&
			\underline{D}_{\mathrm{ext}}^{\mathrm{pre}}(W,\varepsilon,d)
			&:=
			\sup_{w\in W}
			\liminf_{\delta\to0}
			\frac{
				\Phi_{\varepsilon}(w,\delta,d)
			}
			{\log\frac1\delta}.
		\end{aligned}
		\]
		By the order of the supremum and {the inclusions of the relevant sets}, we have
		\[
		\begin{aligned}
			&\overline{D}_{\mathrm{ext}}^{\mathrm{pre}}(W,\varepsilon,d)
			\leq
			\overline{D}_{\mathrm{int}}^{\mathrm{pre}}(W,\varepsilon,d)
			\leq \overline{\mathrm{mdim}}_{M}(W,d),
		\end{aligned}
		\]
		where $\overline{\mathrm{mdim}}_{M}(W,d)$ denotes the upper metric mean dimension of {$W$}.

		Question~(2) asks whether \cite[Theorem 1.3]{Yang-Chen-Zhou2024} remains valid when the supremum over the base point {is} placed outside the limiting process. In particular, a strict gap between
		\(\overline{D}_{\mathrm{ext}}^{\mathrm{pre}}(W,\varepsilon, d)\) and
		\(\overline{D}_{\mathrm{int}}^{\mathrm{pre}}(W,\varepsilon, d)\) is enough to give a negative answer to Question~(2).
		
		Question~(3) asks whether the same interchange of the supremum and the limiting operation remains valid when the
		\(\varepsilon\)-stable set \(\Gamma_\varepsilon^{\{F_n\}_{n\in\mathbb{N}}}(\cdot)\) is replaced by the asymptotic-version of stable set.
		We now make this replacement precise.
		For \(w\in W\), define
		\[
		\Gamma_{\mathrm{as}}(w)
		:=
		\left\{z\in W:\lim_{n\to\infty}d(T^n z,T^n w)=0\right\}.
		\]
		Replacing \(\Gamma_\varepsilon^{\{F_n\}_{n\in\mathbb{N}}}(\cdot)\) by \(\Gamma_{\mathrm{as}}(\cdot)\) in the definitions in \((\ref{def})\) gives the asymptotic counterparts
		\[
		\overline{D}_{\mathrm{int}}^{\mathrm{as}}(W,d),
		\quad
		\underline{D}_{\mathrm{int}}^{\mathrm{as}}(W,d),
		\quad
		\underline{D}_{\mathrm{ext}}^{\mathrm{as}}(W,d),
		\quad
		\overline{D}_{\mathrm{ext}}^{\mathrm{as}}(W,d).
		\]
		By interchanging the order of the supremum and the limiting operations and then by the localization formula for the asymptotic-version of stable set \cite[Theorem~1.5]{Yang-Chen-Zhou2024}, we have
		\[
		\begin{aligned}
			&\overline{D}_{\mathrm{ext}}^{\mathrm{as}}(W, d)
			\leq
			\overline{D}_{\mathrm{int}}^{\mathrm{as}}(W, d)
			= \overline{\mathrm{mdim}}_{M}(W,d).
		\end{aligned}
		\]
		A strict gap between
		\(\overline{D}_{\mathrm{ext}}^{\mathrm{as}}(W, d)\) and
		\(\overline{D}_{\mathrm{int}}^{\mathrm{as}}(W, d)\), along with Proposition~\ref{prop20}~(3), will give a negative answer to Question~(3).

		We will prove that the initially constructed shift space \(X\) also satisfies the analogue of \((\ref{first-step-separation})\)
		for both the preimage and the asymptotic versions of stable sets, giving negative answers to \cite[\S 6, Questions~(2) and~(3)]{Yang-Chen-Zhou2024}. In fact, what we can show is that, after passing to a suitable product system
		\(X\times Y_E\), a complete separation of the corresponding four quantities will be similarly obtained. For details see Proposition \ref{answer} and Remark \ref{last}.
	\end{remark}
	
	Let us now examine the details of this process more closely as follows.
	
	\subsection{The shift space $(X,\sigma)$.} \
	
	We begin by constructing the space \(X\) as a disjoint union of levels
		\(X^{(n)}\). The guiding idea is that different levels $X^{(n)}$ are designed to
		carry different amounts of local topological complexity. More precisely,
		for a point \(x\in X^{(n)}\), the complexity of its \(\varepsilon\)-stable
		set \(\Gamma_\varepsilon(x)\) is governed by the level \(n\). As the scale changes, the level which contributes most to the local entropy
		may also change. This dependence on both \(n\) and scale is the key
		mechanism behind the failure of interchanging the supremum and the
		small-scale \(\limsup\).

Let $([0,2]^{\mathbb Z},\sigma)$ be the shift system, where $\sigma$ denotes the left shift, that is, $(\sigma x)_i = x_{i + 1}$ for all $x= (x_{n})_{n\in\mathbb{Z}}\in [0,2]^{\mathbb Z}$ and any $i\in \mathbb{Z}$. Let $d$ be a compatible metric given by
	\begin{equation}\label{metric}
		d((x_{n})_{n\in\mathbb{Z}},(y_{n})_{n\in\mathbb{Z}})=\sup_{n\in\mathbb{Z}}\frac{|x_{n}-y_{n}|}{2^{|n|}}.
	\end{equation}
	The required $X\subset [0,2]^{\mathbb{Z}}$ is a compact invariant subset constructed as follows.
	First we set
	\[
	X^{(0)}
	= \{0^{\mathbb{Z}}\}
	\;\sqcup\;
	\left\{ x \in \{0,2\}^{\mathbb{Z}} :
	\text{ uniquely }\exists\, i \in \mathbb{Z} \text{ with }
	x_i = 2 \text{ and } x_j = 0 \text{ for all } j \neq i
	\right\}.
	\]
	For each $n\in\mathbb{N}$, we let $G_n$ be a finite equidistant set on $(2^{-n},\,2^{-n+1}]$ with spacing $2^{-n^{2}}$:
	\[
	G_n=\left\{\,2^{-n}+k\cdot2^{-n^{2}}:\ k=1,\cdots,2^{n^{2}-n}\right\} \subset (2^{-n},\,2^{-n+1}] \subset (0, 1].
	\]
	We also fix a strictly increasing sequence $\{P_m\}_{m\in\mathbb{N}}$ in $\mathbb{N}$.
	For every $n \in \mathbb{N}$, put
	\begin{equation}\label{X_n}
		X_{n}=\Bigl\{
		x \in (G_n \cup \{2\})^{\mathbb{Z}} :
		x_i = 2 \text{ whenever } i \equiv 0\;(\mathrm{mod}\, P_n)
		\ \text{and } x_i \in G_n \text{ otherwise}
		\Bigr\}.
	\end{equation}
	Then we define
	\[
	X := \bigcup_{n \ge 0} X^{(n)}\subset [0,2]^{\mathbb{Z}}\ \ \ \ \ \text{where}\ \ \ \ \
	X^{(n)}
	= \bigcup_{i \in \mathbb{Z}} \sigma^{i} X_n
	= \bigsqcup_{i=0}^{P_n - 1} \sigma^{i} X_n\ \text{for each}\ n\in \mathbb{N}.
	\]

	\begin{prop}\label{prop-shiftinv}
		The constructed $(X,\sigma)$ is a subshift in $([0,2]^{\mathbb Z},\sigma)$.
	\end{prop}
	\begin{proof}
		The shift invariance of $X$ (that is, $\sigma (X)= X$) follows directly from that of $X^{(n)}$ for each $n\geq 0$.
		It remains to prove that $X\subset [0, 2]^\mathbb{Z}$ is a closed subset. That is, if we let $\{x^{(m)}\}_{m \in \mathbb{N}} \subset X$ be a sequence which converges to $x$ in $[0, 2]^\mathbb{Z}$, then we need to show $x \in X$.
		
		If infinitely many terms $x^{(m)}$ belong to $X^{(n)}$ for some common $n \ge 0$,
		then $x \in X^{(n)} \subset X$ already by the closeness of $X^{(n)}$ (it is trivial to see that $X^{(n)}\subset [0, 2]^\mathbb{Z}$ is a closed subset).
		
		Otherwise, there exist sequences $m_k \to \infty$ and $n_k \to \infty$ such that
		$x^{(m_k)} \in X^{(n_k)}$.
		We will prove that in this case $x \in X^{(0)}$.
		As $\{x^{(m_k)}\}_{k\in\mathbb{N}}$ converges in $[0, 2]^\mathbb{Z}$,
		for each fixed $i \in \mathbb{Z}$ there is $K_i \in \mathbb{N}$ such that
		either \(x^{(m_k)}_i = 2\) for all $k \ge K_i$, or $x^{(m_k)}_i \in G_{n_k} \subset (2^{-n_k},\,2^{-n_k+1}]\subset (0, 1]$ for all $k \ge K_i$. Note that $2^{-n_k} \searrow 0$ as $k \to \infty$, it follows that
		$x_i \in \{0,2\}$ for every $i \in \mathbb{Z}$.
		In the following, let us show that, if there exists $i_{0}$ such that $x_{i_{0}}=2$, then $x_{j}=0$ for all $j\neq i_{0}$.
		In fact, since $x_{i_0}=2$, we have $x^{(m_k)}_{i_0}=2$ for all $k \ge K_{i_0}$, and then the definition of $X^{(n_k)}$ implies that for all $k \ge K_{i_0}$,
		$x^{(m_k)}_{i_0+j} \in G_{n_k}$ whenever $1 \le |j| \le P_{n_k}-1$ (as $x^{(m_k)} \in X^{(n_k)}$). In particular, for each fixed $l \ne 0$, we have $|l| \le P_{n_k}-1$ and hence $x^{(m_k)}_{i_0+l} \in G_{n_k}\subset (2^{-n_k},\,2^{-n_k+1}]$ for all sufficiently large $k$ (as $P_{n_k}$ diverges to infinity). As $x^{(m_k)}_{i_0+l} \to x_{i_0+l}$ and
		$2^{-n_k} \searrow 0$, one has $x_{i_0+l}=0$. This finishes the proof.
	\end{proof}

	With $(X,\sigma)$ established as a subshift, we next investigate the structure of its stable sets.
	
	\begin{prop}\label{stableset}
		Let $\varepsilon\in(0,1)$ and $x\in X$.
		\begin{enumerate}
			\item If \(x\in X^{(0)}\), then
			\(
			\Gamma_\varepsilon^{\{F_m\}_{m\in\mathbb{N}}}(x)\subset X^{(0)}\)
			and \(\Gamma_{\mathrm{as}}(x)=X^{(0)}\).
			
			\item If \(x\in \sigma^{P_n-i}X_n\subset X^{(n)}\), where
			\(n\in\mathbb N\) and \(0\leq i<P_n\), then
			\[
			\begin{aligned}
				&\Gamma_\varepsilon^{\{F_m\}_{m\in\mathbb{N}}}(x)\subset \sigma^{P_n-i}X_n,
				\\&\Gamma_{\mathrm{as}}(x)
				=
				\left\{
				y\in \sigma^{P_n-i}X_n:
				\exists\, N\in\mathbb Z \text{ such that } x_k=y_k
				\text{ for all } k\geq N
				\right\}.
			\end{aligned}
			\]
			Moreover, \(\Gamma_\varepsilon^{\{F_m\}_{m\in\mathbb{N}}}(x)=\sigma^{P_n-i}X_n\) whenever
			\(\varepsilon\geq 2^{-n}\).
		\end{enumerate}
	\end{prop}
	\begin{proof}
		We note by the definition \eqref{metric} of the metric $d$ that, for any $x=(x_{n})_{n\in \mathbb{Z}}\in X$,
		\begin{equation}\label{def-stable}
			\begin{aligned}
				&\Gamma_{\varepsilon}^{\{F_m\}_{m\in\mathbb{N}}}(x)\subset\left\{y\in X: |x_{l}-y_{l}|\leq \varepsilon \text{ for all }l\geq 0\right\},
				\\ &\Gamma_{\mathrm{as}}(x)=\left\{y\in X: \lim_{l\to\infty}|x_{l}-y_{l}|=0\right\}.
			\end{aligned}
		\end{equation}
		
		We assume firstly that $x\in X^{(0)}$. Then the sequence $x$ contains the symbol $2$ at most once, and all other coordinates are equal to $0$. For any \(y\in X^{(0)}\), there exists  \(N\in\mathbb N\) such that
		\(x_m=y_m=0\) for all \(m\geq N\), and hence $X^{(0)}\subset \Gamma_{\mathrm{as}}(x)$.
		On the other hand, if $y\in X^{(n)}$ for some $n\in\mathbb{N}$, then the symbol $2$ occurs infinitely many times in $y$, and hence there exist infinitely many indices $m\in\mathbb{N}$ with
		$|x_m-y_m|=2>\varepsilon$.
		Therefore, by \eqref{def-stable}, this shows that $\Gamma_{\varepsilon}^{\{F_k\}_{k\in\mathbb{N}}}(x)\subset X^{(0)}$  and $\Gamma_{\mathrm{as}}(x)=X^{(0)}$.
		
		Now suppose $x\in \sigma^{P_n-i}X_n$ for some
		$i\in\{0,1, \cdots,P_n-1\}$.
		By \eqref{X_n} and \eqref{def-stable} one has
		\[
		\Gamma^{\{F_k\}_{k\in\mathbb{N}}}_{\varepsilon}(x)
		\subset
		\Bigl\{
		y\in(G_n\cup\{2\})^{\mathbb Z} :
		\begin{aligned}
			&	y_j=2 \text{ for } j\ge 0\ \text{such that}\ j \equiv i~(\text{mod } P_n) \\
			& |y_j-x_j|\le\varepsilon \text{ for all other $j\geq 0$}
		\end{aligned}
		\Bigr\},
		\]
		as $|2-a|\ge 1>\varepsilon$ for all $a\in G_n$.
		Assume that $z\in X^{(m)}$ for some $m\ne n$.
		If $m=0$, then there exist infinitely many indices $l\in\mathbb{N}$ with
		$|x_l-z_l|=2>\varepsilon$, and so $z\notin \Gamma^{\{F_k\}_{k\in\mathbb{N}}}_{\varepsilon}(x)\cup \Gamma_{\mathrm{as}}(x)$ follows from (\ref{def-stable}).
		If $m\in\mathbb{N}$, as the symbol $2$ occurs with different periods in $x$ and $z$,
		there exist infinitely many indices $l\in\mathbb{N}$ such that, either
		$x_l=2$ and $z_l\in G_m$ or
		$x_l\in G_n$ and $z_l=2$, hence $|x_l-z_l|\ge 1>\varepsilon$, and so $z\notin \Gamma^{\{F_k\}_{k\in\mathbb{N}}}_{\varepsilon}(x)\cup \Gamma_{\mathrm{as}}(x)$.
		Therefore,
		\[
		\Gamma^{\{F_k\}_{k\in\mathbb{N}}}_{\varepsilon}(x)\cup \Gamma_{\mathrm{as}}(x)\subset X\backslash\bigsqcup_{m\neq n} X^{(m)}=X^{(n)}.
		\]
		Moreover, from the defining positions of the symbol $2$,
		we could conclude
		\[
		\Gamma^{\{F_k\}_{k\in\mathbb{N}}}_{\varepsilon}(x)\cup \Gamma_{\mathrm{as}}(x)\subset \sigma^{P_n-i}X_n.
		\] Using again (\ref{def-stable}) and by the discreteness of $\sigma^{P_n-i}X_n$ we easily obtain
		\[
		\Gamma_{\mathrm{as}}(x)=\left\{
		y\in \sigma^{P_n-i}X_n:
		\exists\, N\in\mathbb Z \text{ such that } x_k=y_k
		\text{ for all } k\geq N
		\right\}.
		\]
		
		If additionally we assume $\varepsilon \ge 2^{-n}$, then for any
		$y\in \sigma^{P_n-i}X_n$ we have that, for each $k\in \mathbb{Z}$, either $x_k = y_k = 2$ or both $x_k$ and $y_k$ are contained in $G_n$, in particular,
		\(
		|x_k-y_k|\le 2^{-n}\le \varepsilon,
		\)
		and hence
		$
		\sigma^{P_n-i}X_n \subset \Gamma^{\{F_k\}_{k\in\mathbb{N}}}_{\varepsilon}(x),
		$
		consequently
		\(
		\Gamma^{\{F_k\}_{k\in\mathbb{N}}}_{\varepsilon}(x)=\sigma^{P_n-i}X_n.
		\)
	\end{proof}
	
	We now estimate for each \(\varepsilon\in(0,1)\) the quantities
	\(\overline{D}_{\mathrm{int}}(X,\varepsilon,d),\)
	\(\underline{D}_{\mathrm{int}}(X,\varepsilon,d)\),
	\(\underline{D}_{\mathrm{ext}}(X,\varepsilon,d)\) and
	\(\overline{D}_{\mathrm{ext}}(X,\varepsilon,d)\), together with their preimage and asymptotic stable-set counterparts.

	\begin{prop}\label{G=0}
		\(
		\overline{D}_{\mathrm{int}}(X,\varepsilon,d)
		=
		\underline{D}_{\mathrm{int}}(X,\varepsilon,d)
		= 1 > 0 =
		\overline{D}_{\mathrm{ext}}(X,\varepsilon,d)
		=
		\underline{D}_{\mathrm{ext}}(X,\varepsilon,d).
		\)
		The same conclusion holds for both the preimage and asymptotic counterparts.
	\end{prop}
	\begin{proof}
		For convenience we set $\Gamma (x) = \Gamma_{\varepsilon}^{\{F_n\}_{n\in\mathbb{N}}}(x)\cup \Gamma_{\mathrm{as}}(x)$ for each $x\in X$.
		
		To prove \(\overline{D}_{\mathrm{ext}}(X,\varepsilon,d)
		= \overline{D}_{\mathrm{ext}}^{\mathrm{pre}}(X,\varepsilon,d)
		=\overline{D}_{\mathrm{ext}}^{\mathrm{as}}(X,d)=0\),
		it suffices to prove
		\begin{equation}\label{perpoint1}
			\limsup_{\delta\to 0}
			\frac{h_{\mathrm{top}}(\Gamma (x),\delta,\{F_n\}_{n\in\mathbb N},d)}
			{\log\frac{1}{\delta}}
			=0\quad\text{and}\quad \limsup_{\delta\to 0}\frac{\Phi_{\varepsilon}(x,\delta,d)}{\log \frac{1}{\delta}}=0
		\end{equation}
		for every $x\in X$. We shall verify that for each $x\in X$
		\begin{equation*}\label{finite-ent}
			\infty>	h_{\mathrm{top}}(\Gamma (x),\{F_n\}_{n\in\mathbb N})=\sup_{\delta>0}h_{\mathrm{top}}(\Gamma (x),\delta,\{F_n\}_{n\in\mathbb N},d)\ \text{and}\
			\infty>\sup_{\delta>0} \Phi_{\varepsilon}(x,\delta,d).
		\end{equation*}
		Assume $x\in X^{(n)}$ for some $n\geq 0$, then \(
		\Gamma (x)\subset X^{(n)}
		\subset (G_n\cup\{0,2\})^{\mathbb Z}
		\)
		by Proposition~\ref{stableset}  (set $G_{0}=\emptyset$ by convention). Moreover, due to the shift invariance of $X^{(n)}$, we have for all $k\in\mathbb{Z}$, $\sigma^{-k}\Gamma^{\{F_m\}_{m\in\mathbb{N}}}_{\varepsilon}(x)\subset X^{(n)}
		\subset (G_n\cup\{0,2\})^{\mathbb Z}$.
		Hence, by the definition of $\Phi_{\varepsilon}(x,\delta,d)$,
		\begin{equation}\label{estimateGn}
			\begin{aligned}
				&h_{\mathrm{top}}(\Gamma (x),\{F_n\}_{n\in\mathbb N})\le
				h_{\mathrm{top}}\left((G_n\cup\{0,2\})^{\mathbb Z},\{F_n\}_{n\in\mathbb N}\right)	=
				\log(|G_n|+2)< \infty.
				\\&\sup_{\delta>0} \Phi_{\varepsilon}(x,\delta,d)\leq \sup_{\delta>0}\limsup_{n\to\infty}
				\frac{1}{|F_n|}
				\log s_{F_n}
				\bigl(
				(G_n\cup\{0,2\})^{\mathbb Z},
				\delta,d
				\bigr)	=
				\log(|G_n|+2)< \infty.
			\end{aligned}
		\end{equation}
		
		Now let us prove  \(\min\left\{\underline{D}_{\mathrm{int}}(X,\varepsilon,d),\underline{D}_{\mathrm{int}}^{\mathrm{pre}}(X,\varepsilon,d),\underline{D}_{\mathrm{int}}^{\mathrm{as}}(X,d)
		\right\}
		\geq 1\).
		Fix $n\in\mathbb{N}$ sufficiently large such that $2^{-n}\leq \varepsilon$. Given any $\delta\leq 2^{-(n+1)^{2}}$, there exists $m=m_{\delta}\geq n$ such that
		\begin{equation}\label{delta-bound}
			\delta\in (2^{-(m+2)^{2}},2^{-(m+1)^{2}}]\ \ \text{ (here $m$ depends on the parameter $\delta$)}.
		\end{equation}
		Pick arbitrarily a point $x_{\delta}\in X_{m}$ and fix it.
		\begin{enumerate}
			
			\item
			By $\varepsilon\geq 2^{-n} \geq2^{-m}$ and Proposition \ref{stableset} (2) one has
			\(
			\Gamma^{\{F_l\}_{l\in\mathbb{N}}}_{\varepsilon}(x_{\delta})=X_{m}.
			\)
			
			\item For each $k\in\mathbb N$, let $\mathcal{C}_{m,k}$ denote the collection of all
			cylinder sets in $X_m$ obtained by fixing the coordinates on
			\(
			F_{kP_m}=\{0,1,\cdots,kP_m-1\}.
			\)
			Every such cylinder has length $kP_m$ and
			\(
			|\mathcal{C}_{m,k}|
			=
			|G_m|^{k(P_m-1)},
			\)
			since among these coordinates exactly $k$ positions are fixed to be the symbol $2$,
			while each of the remaining $k(P_m-1)$ positions may take arbitrary values in $G_m$. We also note that, by Proposition~\ref{stableset} (2),
			\[
			\begin{aligned}
				\Gamma_{\mathrm{as}}(x_\delta) \supset	\left\{
				y\in X_m:
				(x_\delta)_j=y_j
				\text{ for all } j\geq kP_m
				\right\},
			\end{aligned}
			\]
			then \(	\Gamma_{\mathrm{as}}(x_\delta)\) has a non-empty intersection with each element of \(\mathcal{C}_{m,k}\). For any \(C\in\mathcal{C}_{m,k}\), choose a point from
			$
			\Gamma_{\mathrm{as}}(x_\delta)\cap C
			$
			and denote the resulting set by \(E_{m,k}\).
			By the construction of $G_m$, $E_{m,k}\subset \Gamma_{\mathrm{as}}(x_\delta)\subset X_m$ is $(F_{kP_m},2^{-m^2-1})$-separated.
		\end{enumerate}
		By (\ref{delta-bound}) we have $\delta\leq 2^{-m^{2}-1}$, and then together with the  $\sigma^{P_m}$-invariance of $X_{m}$ one obtains
		\begin{equation*}\label{k-estimate}
			\begin{aligned}
				&s_{F_{kP_{m}}}(\Gamma^{\{F_l\}_{l\in\mathbb{N}}}_{\varepsilon}(x_{\delta}),\delta,d)=s_{F_{kP_{m}}}(X_{m},\delta,d)\geq s_{F_{kP_{m}}}(X_{m},2^{-m^{2}-1},d)\geq|E_{m, k}|= |G_{m}|^{k(P_{m}-1)},
				\\&s_{F_{kP_{m}}}(\sigma^{-kP_m}\Gamma^{\{F_l\}_{l\in\mathbb{N}}}_{\varepsilon}(x_{\delta}),\delta,d)=s_{F_{kP_{m}}}(X_{m},\delta,d)\geq  |G_{m}|^{k(P_{m}-1)},
				\\&s_{F_{kP_{m}}}(\Gamma_{\mathrm{as}}(x_{\delta}),\delta,d)\geq s_{F_{kP_{m}}}(\Gamma_{\mathrm{as}}(x_{\delta}),2^{-m^{2}-1},d)\geq |E_{m,k}|= |G_{m}|^{k(P_{m}-1)}.
			\end{aligned}
		\end{equation*}
		Consequently, combined with (\ref{delta-bound}) we obtain
		\[
		\begin{aligned}
			&\frac{\min\left\{h_{\mathrm{top}}(\Gamma^{\{F_n\}_{n\in\mathbb{N}}}_{\varepsilon}(x_{\delta}),\delta,\{F_n\}_{n\in\mathbb{N}},d),\Phi_{\varepsilon}(x_{\delta},\delta,d),h_{\mathrm{top}}(\Gamma_{\mathrm{as}}(x_{\delta}),\delta,\{F_n\}_{n\in\mathbb{N}},d)\right\}}{\log \frac{1}{\delta}}
			\\\geq&
			\frac{\left(1-\frac{1}{P_{m}}\right)\cdot\log|G_{m}|}{\log \frac{1}{\delta}}
			\geq
			\frac{\left(1-\frac{1}{P_{m}}\right)(m^{2}-m)}{(m+2)^2}.
		\end{aligned}
		\]
		Since $m = m_\delta\to\infty$ as $\delta\to 0$
		and $\{P_m\}_{m\in\mathbb{N}}\subset \mathbb{N}$ is strictly increasing,
		we conclude that
		\begin{equation*}
			\min\left\{\underline{D}_{\mathrm{int}}(X,\varepsilon,d),\underline{D}_{\mathrm{int}}^{\mathrm{pre}}(X,\varepsilon,d),\underline{D}_{\mathrm{int}}^{\mathrm{as}}(X,d)
			\right\}
			\geq  1.
		\end{equation*}

		It remains to prove
		\(\max \left\{	\overline{D}_{\mathrm{int}}(X,\varepsilon,d),\overline{D}_{\mathrm{int}}^{\mathrm{pre}}(X,\varepsilon,d),\overline{D}_{\mathrm{int}}^{\mathrm{as}}(X,d)\right\}\le 1\) by establishing
		\begin{equation*}\label{sup-estimate}
			\max\left\{\sup_{x\in X}h_{\mathrm{top}}(\Gamma (x),\delta,\{F_n\}_{n\in\mathbb{N}},d),\sup_{x\in X}\Phi_{\varepsilon}(x,\delta,d)\right\}
			\leq
			\log \frac{4}{\delta}\ \ \text{for each $\delta\in(0,\frac{1}{2})$}.
		\end{equation*}

		Let $x\in X$. The above inequality holds trivially for $x\in X^{(0)}$, since by Proposition~\ref{stableset}
		\[
		\begin{aligned}
			&h_{\mathrm{top}}(\Gamma (x),\delta,\{F_n\}_{n\in\mathbb{N}},d)
			\leq h_{\mathrm{top}}(\{0,2\}^{\mathbb{Z}},\{F_n\}_{n\in\mathbb{N}})=\log 2\leq
			\log \frac{1}{\delta},
			\\&\Phi_{\varepsilon}(x,\delta,d)\leq 	 h_{\mathrm{top}}(\{0,2\}^{\mathbb{Z}},\{F_n\}_{n\in\mathbb{N}})=\log 2\leq \log \frac{1}{\delta}.
		\end{aligned}
		\]
		Hence it remains to consider the case where
		$x\in \sigma^{i}X_m$ for some $m\in\mathbb{N}$ and $i\in\{0, 1, \cdots,$ $P_m-1\}$.
		In this situation, Proposition \ref{stableset} yields
		\(\Gamma (x)\subset \sigma^{i}X_m\subset (G_m\cup\{2\})^{\mathbb Z}
		\) and $\sigma^{-p}\Gamma^{\{F_n\}_{n\in\mathbb{N}}}_{\varepsilon}(x)\subset \sigma^{i-p}X_{m}\subset (G_m\cup\{2\})^{\mathbb Z}$ for every $p\in\mathbb{Z}$.
		We distinguish two cases.
		
		If $2^{m^{2}-m}<\frac{1}{\delta}$, then by \eqref{estimateGn} one has
		\[
		\max \left\{h_{\mathrm{top}}(\Gamma (x),\delta,\{F_n\}_{n\in\mathbb{N}},d), \Phi_{\varepsilon}(x,\delta,d)\right\}
		\le	\log(|G_m|+1)
		=
		\log\bigl(2^{m^{2}-m}+1\bigr)
		<
		\log\frac{2}{\delta}.
		\]
		
		Now we assume $2^{m^{2}-m}\ge \frac{1}{\delta}$, and set
		\(
		M=\lfloor \delta\cdot 2^{m^{2}-m}\rfloor+1\ge \delta\cdot 2^{m^{2}-m}.
		\)
		Define
		\[
		G_{m,M}
		=
		\left\{
		2^{-m}+Mk\cdot 2^{-m^{2}}
		:
		k=1,\cdots,\left\lfloor\frac{2^{m^{2}-m}}{M}\right\rfloor
		\right\}
		\subset G_m,
		\]
		where $\lfloor \frac{2^{m^{2}-m}}{M}\rfloor\geq 1$ follows from the assumption that $\delta<\frac{1}{2}$ and $2^{m^{2}-m}\ge \frac{1}{\delta}$ and the construction of \(
		M=\lfloor \delta\cdot 2^{m^{2}-m}\rfloor+1.
		\)
		We also introduce in $X_m$ the following closed subset
		\[
		X_{m,M}
		=
		\Bigl\{
		x \in (G_{m,M}\cup\{2\})^{\mathbb{Z}}
		:
		x_i=2 \text{ whenever } i\equiv 0\  (\mathrm{mod}\ P_m)
		\ \text{and } x_i\in G_{m,M} \text{ otherwise}
		\Bigr\}.
		\]

		For every $l\in\mathbb{N}$ and $j\in\mathbb{Z}$, let $\mathcal{C}^{j}_{m,M,l}$ be the collection of all cylinders
		in $\sigma^{j}X_{m,M}$ determined by fixing the coordinates on
		$\{-m^{2},\cdots,l-1\}$.
		Choose one point from each cylinder in $\mathcal{C}^{j}_{m,M,l}$, and denote the resulting set by $E^{j}_{m,M,l}$.
		We claim that for every $p\in\mathbb{N}$, the set $E^{j}_{m,M,p+m^{2}-m}$ is an $(F_p,2\delta)$-spanning set for $\sigma^{j}X_m$.
		In fact, fix arbitrarily given $p\in\mathbb{N}$ and $y\in \sigma^{j}X_m$, it is trivial to see that there exists $z\in E^{j}_{m,M,p+m^{2}-m}$ such that
		\(
		|z_k-y_k|\le M\cdot 2^{-m^{2}}\le 2\delta
		\)
		for all $-m^{2}\leq k<p+m^{2}-m$. Therefore, for every $n\in F_p$,
		\[
		d(\sigma^{n}z,\sigma^{n}y)
		\le
		\max\bigl\{2\delta,\,2^{-(p+m^{2}-m-1-n)},\,2^{-m^{2}-n}\bigr\}
		\le
		\max\bigl\{2\delta,\,2^{-(m^{2}-m)},\,2^{-m^{2}}\bigr\}
		\le 2\delta.
		\]
		In particular, $E^{i}_{m,M,p+m^{2}-m}$ is an $(F_p,2\delta)$-spanning set of $\sigma^{i}X_m$ and hence for $\Gamma (x)$, and $E^{i-p}_{m,M,p+m^{2}-m}$ is an $(F_p,2\delta)$-spanning set of $\sigma^{i-p}X_m$ and hence for $\sigma^{-p}\Gamma^{\{F_n\}_{n\in\mathbb{N}}}_{\varepsilon}(x)$.
		
		Note that \(|E_{m,M,l}^{j}|	\le |G_{m,M}|^{l+m^{2}}\) for every $l\in \mathbb{N}$ and $j\in\mathbb{Z}$, therefore by \eqref{se-sp} one has
		\begin{eqnarray*}
			& & \max\left\{	h_{\mathrm{top}}(\Gamma(x),4\delta,\{F_n\}_{n\in\mathbb{N}},d),\Phi_{\varepsilon}(x,4\delta,d)\right\}
			\\ & \leq & \limsup\limits_{p\to\infty}\frac{\log |G_{m,M}|^{p+2m^{2}-m}}{|F_p|}
			=
			\log |G_{m,M}|\le
			\log \frac{2^{m^{2}-m}}{M}\leq \log \frac{1}{\delta},
		\end{eqnarray*}
		and hence		$
		\max \left\{	\overline{D}_{\mathrm{int}}(X,\varepsilon,d),\overline{D}_{\mathrm{int}}^{\mathrm{pre}}(X,\varepsilon,d),\overline{D}_{\mathrm{int}}^{\mathrm{as}}(X,d)\right\}
		\le 1.
		$
		This finishes the proof.
	\end{proof}

	\subsection{The auxiliary system $(Y_E,\sigma)$}\

	We now construct the auxiliary shift space $(Y_E,\sigma)$. Let $E\subset \mathbb{Z}$ and define
	\[
	\Lambda_{E}=\{x\in \{0,1\}^{\mathbb{Z}}:x_{i}=0 \text{ whenever }i\in E\}.
	\]
	Equipped with the metric $d$ given by (\ref{metric}), $\Lambda_E$ is a compact metric space.
	Set $
	Y_{E}=\Lambda_{E}^{\mathbb{Z}}$, which is again a compact metric space endowed with the metric $\rho$ given by
	\[
	\rho\bigl((x^{(n)})_{n\in\mathbb{Z}},(y^{(n)})_{n\in\mathbb{Z}}\bigr)
	=\sup_{m\in\mathbb{Z}}\frac{d(x^{(m)},y^{(m)})}{2^{|m|}},
	\quad \ \ \text{where}\ x^{(m)},y^{(m)}\in \Lambda_E\ \text{for each}\ m\in\mathbb{Z}.
	\]
	
	For each \(c\geq 1\), we also define the metric $\rho_{c}$ corresponding to $\rho$ by
	\[
	\rho_c\bigl((x^{(n)})_{n\in\mathbb Z},(y^{(n)})_{n\in\mathbb Z}\bigr)
	:=
	\left(
	\rho\bigl((x^{(n)})_{n\in\mathbb Z},(y^{(n)})_{n\in\mathbb Z}\bigr)
	\right)^\frac{1}{c} .
	\]
	Obviously the metric \(\rho_c\) is compatible with the product topology on \(Y_E\). In particular, \(\rho_1=\rho\). Since  several metrics \(\rho_c\) will be used in this subsection, we shall make the metric explicit in the notation for stable sets. More precisely, we write
	$
	\Gamma_{\varepsilon}^{\{F_n\}_{n\in\mathbb N}}(y,\rho_c)
	\,\text{and}\,
	\Gamma_{\mathrm{as}}(y,\rho_c)
	$
	for the corresponding \(\varepsilon\)-stable set and asymptotic stable set with respect to the metric \(\rho_c\).
	
	In what follows, we record some necessary properties of the $\mathbb{Z}$-action $(Y_E,\sigma)$.
	
	\begin{prop}\label{YE}
		Let \(E\subset \mathbb{Z}\) and equip the system \((Y_E,\sigma)\) with above metric \(\rho_c\). Then:
		\begin{enumerate}
			\item The upper and lower metric mean dimensions are given respectively by
			\[
			\begin{aligned}
				&\overline{\mathrm{mdim}}_M(Y_{E},\{F_{n}\}_{n\in\mathbb{N}},\rho_{c})=
c{\cdot\limsup_{n\to\infty}\frac{\left|\{-n+1,-n+2,\cdots,n-1\}\backslash E\right|}{n}},
				\\&\underline{\mathrm{mdim}}_M(Y_{E},\{F_{n}\}_{n\in\mathbb{N}},\rho_{c})=
c{\cdot\liminf_{n\to\infty}\frac{\left|\{-n+1,-n+2,\cdots,n-1\}\backslash E\right|}{n}}.
			\end{aligned}
			\]
			
			\item For every  finite subset $F\subset \mathbb{Z}$,
			\[
			\begin{aligned}
				\overline{\mathrm{mdim}}_M(Y_{E\cup F},\{F_{n}\}_{n\in\mathbb{N}},\rho_{c})&=\overline{\mathrm{mdim}}_M(Y_{E},\{F_{n}\}_{n\in\mathbb{N}},\rho_{c}),
				\\\underline{\mathrm{mdim}}_M(Y_{E\cup F},\{F_{n}\}_{n\in\mathbb{N}},\rho_{c})&=\underline{\mathrm{mdim}}_M(Y_{E},\{F_{n}\}_{n\in\mathbb{N}},\rho_{c}).
			\end{aligned}		
			\]
			
			\item For every $\varepsilon>0$,
			\[
			\begin{aligned}
				\overline{D}_{\mathrm{int}}(Y_{E},\varepsilon,\rho_{c})&=\overline{D}_{\mathrm{ext}}(Y_{E},\varepsilon,\rho_{c})=\overline{D}_{\mathrm{int}}^{\mathrm{pre}}(Y_{E},\varepsilon,\rho_{c})
				=\overline{D}^{\mathrm{pre}}_{\mathrm{ext}}(Y_{E},\varepsilon,\rho_{c})
				\\&=\overline{D}_{\mathrm{int}}^{\mathrm{as}}(Y_{E},\rho_{c})=\overline{D}^{\mathrm{as}}_{\mathrm{ext}}(Y_{E},\rho_{c})
				=\overline{\mathrm{mdim}}_M(Y_{E},\{F_{n}\}_{n\in\mathbb{N}},\rho_{c}),
				\\				\underline{D}_{\mathrm{int}}(Y_{E},\varepsilon,\rho_{c})&=\underline{D}_{\mathrm{ext}}(Y_{E},\varepsilon,\rho_{c})=
				\underline{D}_{\mathrm{int}}^{\mathrm{pre}}(Y_{E},\varepsilon,\rho_{c})=\underline{D}^{\mathrm{pre}}_{\mathrm{ext}}(Y_{E},\varepsilon,\rho_{c})
				\\&=\underline{D}_{\mathrm{int}}^{\mathrm{as}}(Y_{E},\rho_{c})=\underline{D}^{\mathrm{as}}_{\mathrm{ext}}(Y_{E},\rho_{c})
				=\underline{\mathrm{mdim}}_M(Y_{E},\{F_{n}\}_{n\in\mathbb{N}},\rho_{c}).
			\end{aligned}
			\]
		\end{enumerate}
	\end{prop}
	
	\begin{proof}
		(1) We prove only the first equality, as proof of the second one is entirely analogous.
		
		For every $\delta>0$, by exploiting the product structure of $Y_E$ along with (\ref{se-sp}), one has
		\begin{equation}\label{se-sp-YE}
				s_{\{0\}}(\Lambda_{E},\delta,d)^{|F_n|}
				\leq s_{F_{n}}(Y_E,\delta,\rho)
				\leq r_{F_{n}}(Y_E,\dfrac{\delta}{2},\rho)
				\leq r_{\{0\}}(\Lambda_{E},\frac{\delta}{2},d)^{|\widetilde{F}_{n_{\delta}}|}
				\leq s_{\{0\}}(\Lambda_{E},\frac{\delta}{2},d)^{|\widetilde{F}_{n_{\delta}}|},
			\end{equation}
			where $\widetilde{F}_{n_\delta}= \{-\big\lceil \frac{\log\frac{1}{\delta}}{\log 2}\big\rceil,-\big\lceil \frac{\log\frac{1}{\delta}}{\log 2}\big\rceil+1,\cdots,n+\big\lceil \frac{\log\frac{1}{\delta}}{\log 2}\big\rceil\}$. This shows that the upper metric mean dimension of $Y_E$ can be computed directly from the scale complexity of $\Lambda_E$, namely,
		\begin{equation}\label{sp for YE}
			\overline{\mathrm{mdim}}_M(Y_{E},\{F_{n}\}_{n\in\mathbb{N}},\rho)
			=
			\limsup_{\delta \to 0}\frac{\log s_{\{0\}}(\Lambda_{E},\delta,d)}{\log \frac{1}{\delta}}=	\limsup_{\delta \to 0}\frac{\log r_{\{0\}}(\Lambda_{E},\delta,d)}{\log \frac{1}{\delta}}.
		\end{equation}
		
		For each $\delta>0$, set
		\(
		m_{\delta}=\Big\lfloor\frac{\log \frac{1}{\delta}}{\log 2}\Big\rfloor
		\)
		(and then $2^{m_\delta}\le \frac{1}{\delta}\le 2^{m_\delta+ 1}$), hence
		\begin{equation}\label{estimate-YE}
				\begin{aligned}
					2^{|\{-m_{\delta}+1,-m_{\delta}+2,\cdots,m_{\delta}-1\}\backslash E|}
					&\leq s_{\{0\}}(\Lambda_{E},2^{-m_{\delta}},d)
					\leq s_{\{0\}}(\Lambda_{E},\delta,d)
					\\&	\leq s_{\{0\}}(\Lambda_{E},2^{-m_{\delta}-1},d)
					\leq 2^{|\{-m_{\delta}-2,-m_{\delta}-1,\cdots,m_{\delta}+2\}\backslash E|}.
				\end{aligned}
		\end{equation}
		On one hand, applying \eqref{sp for YE} to the latter part of the inequality \eqref{estimate-YE} one has
		\[
		\overline{\mathrm{mdim}}_M(Y_{E},\{F_{n}\}_{n\in\mathbb{N}},\rho)
		\leq
		\limsup_{n\to\infty}\frac{\left|\{-n+1,-n+2,\cdots,n-1\}\backslash E\right|}{n}.
		\]
		For the reverse direction, we also use (\ref{estimate-YE}). Evaluating at the dyadic scales $\delta=2^{-n}$ gives
		\[
		\begin{aligned}
		\limsup_{n\to\infty}\frac{\left|\{-n+1,-n+2,\cdots,n-1\}\backslash E\right|}{n}
				\leq
				\overline{\mathrm{mdim}}_M(Y_{E},\{F_{n}\}_{n\in\mathbb{N}},\rho).
		\end{aligned}
		\]
		
		Therefore, it is proven the first equality in (1) under metric $\rho=\rho_1$:
		\[
			\overline{\mathrm{mdim}}_M(Y_{E},\{F_{n}\}_{n\in\mathbb{N}},\rho)
			=
			\limsup_{n\to\infty}\frac{\left|\{-n+1,-n+2,\cdots,n-1\}\backslash E\right|}{n}.
			\]
		Now let \(F\subset\mathbb Z\) and \(c> 1\), \(\delta>0\), we have
		$
		s_F(Y_E,\delta,\rho_c)
		=
		s_F(Y_E,\delta^c,\rho),
		$
		and then
		\begin{equation}\label{newmetric}
			h_{\mathrm{top}}(Y_E,\delta,\{F_n\}_{n\in\mathbb N},\rho_c)
			=
			h_{\mathrm{top}}(Y_E,\delta^c,\{F_n\}_{n\in\mathbb N},\rho).
		\end{equation}
		It follows that
		\begin{equation*}\label{times c}
			\begin{aligned}
				\overline{\mathrm{mdim}}_M
				(Y_E,\{F_n\}_{n\in\mathbb N},\rho_c)
				&=
				c\cdot
				\overline{\mathrm{mdim}}_M
				(Y_E,\{F_n\}_{n\in\mathbb N},\rho) \\
				&=
				c\cdot
				\limsup_{n\to\infty}\frac{\left|\{-n+1,-n+2,\cdots,n-1\}\backslash E\right|}{n}.
			\end{aligned}
		\end{equation*}

		(2) The case of $c=1$
		comes directly from the following easy estimations and (\ref{sp for YE}):
		\[
		s_{\{0\}}(\Lambda_{E\cup F},\delta,d)
		\leq
		s_{\{0\}}(\Lambda_E,\delta,d)
		\leq
		r_{\{0\}}(\Lambda_E,\frac{\delta}{2},d)
		\leq
		2^{|F|}\,r_{\{0\}}(\Lambda_{E\cup F},\frac{\delta}{2},d).
		\]
		Now the conclusion for the case of \(c>1\) follows immediately from (1).
		
		(3) Again we only prove the first ones. Fix each $\varepsilon>0$. By (\ref{f1}) and \cite[Theorem 1.3 and Theorem 1.5]{Yang-Chen-Zhou2024} we already know that
		\[
		\begin{aligned}
			&\overline{D}_{\mathrm{ext}}(Y_{E},\varepsilon,\rho_{c})
			\leq
			\overline{D}_{\mathrm{int}}(Y_{E},\varepsilon,\rho_{c})
			=
			\overline{\mathrm{mdim}}_M(Y_{E},\{F_{n}\}_{n\in\mathbb{N}},\rho_{c}),
			\\&\overline{D}_{\mathrm{ext}}^{\mathrm{pre}}(Y_{E},\varepsilon,\rho_{c})
			\leq
			\overline{D}_{\mathrm{int}}^{\mathrm{pre}}(Y_{E},\varepsilon,\rho_{c})
			\leq
			\overline{\mathrm{mdim}}_M(Y_{E},\{F_{n}\}_{n\in\mathbb{N}},\rho_{c}),
			\\&\overline{D}_{\mathrm{ext}}^{\mathrm{as}}(Y_{E},\rho_{c})
			\leq
			\overline{D}_{\mathrm{int}}^{\mathrm{as}}(Y_{E},\rho_{c})
			=
			\overline{\mathrm{mdim}}_M(Y_{E},\{F_{n}\}_{n\in\mathbb{N}},\rho_{c}).
		\end{aligned}
		\]
		It remains to prove
		\[		\min \left\{\overline{D}_{\mathrm{ext}}(Y_{E},\varepsilon,\rho_{c}),\overline{D}_{\mathrm{ext}}^{\mathrm{pre}}(Y_{E},\varepsilon,\rho_{c}),\overline{D}_{\mathrm{ext}}^{\mathrm{as}}(Y_{E},\rho_{c})\right\}
		\geq
		\overline{\mathrm{mdim}}_M(Y_{E},\{F_{n}\}_{n\in\mathbb{N}},\rho_{c}).
		\]

Set \(m=m_{\varepsilon,c} =
			\max\Big\{\Big\lceil\frac{c\cdot\log\frac{1}{\varepsilon}}{\log 2}\Big\rceil,0\Big\}
			\), then \(2^{m+ 1} \varepsilon^c > 1\). For every $y\in Y$,
			\begin{eqnarray}
				\Gamma^{\{F_n\}_{n\in\mathbb{N}}}_{\varepsilon}(y,\rho_{c})
				& \hskip -6pt = & \hskip -6pt \left\{z=(z^{(n)})_{n\in\mathbb{Z}}\in Y_{E}:\rho(\sigma^{n}(y),\sigma^{n}(z))\leq \varepsilon^{c}\text{ for all }n\geq 0\right\}
				\nonumber		\\ & \hskip -6pt  = & \hskip -6pt \left\{z=(z^{(n)})_{n\in\mathbb{Z}}\in Y_{E}:d(y^{(k+n)},z^{(k+n)})\leq 2^{|k|}\varepsilon^{c}\text{ for all $n\geq 0$ and $k\in\mathbb{Z}$}\right\}
				\nonumber	\\ & \hskip -6pt  =& \hskip -6pt  \left\{z=(z^{(n)})_{n\in\mathbb{Z}}\in Y_{E}:d(y^{(k+n)},z^{(k+n)})\leq 2^{|k|}\varepsilon^{c}\text{ for all $n\geq 0$ and $|k|\leq m$}\right\}
	\nonumber \\
				& \hskip -6pt  \supset & \hskip -6pt
				\left\{
				z=(z^{(n)})_{n\in\mathbb{Z}}\in Y_{E}:
				d(y^{(n)},z^{(n)})\leq \varepsilon^{c}\text{ for all $n\geq -m$}
				\right\}. 			\label{finite inv}	
			\end{eqnarray}
			Let $y_{\varepsilon}=(y^{(n)})_{n\in\mathbb{Z}}$ in $Y_E=\Lambda_E^{\mathbb{Z}}$ be such that
			\(y^{(n)}_{i}=0\) once $-m\leq i\leq m$ and $n\in\mathbb{Z}$. Then
			\begin{equation}\label{stable-YE}
				\Gamma^{\{F_n\}_{n\in\mathbb{N}}}_{\varepsilon}(y,\rho_{c})
				\supset Y_{E\cup\{-m,-m+1,\cdots, m\}}.
			\end{equation}
			Indeed, if $z^{(n)}_{i}=0$ for all $-m\le i\le m$ and every $n\in\mathbb{Z}$, then
			$d(z^{(n)},y^{(n)})\le 2^{-(m+1)}<\varepsilon^{c}$ for every $n\in\mathbb{Z}$, and hence $z\in\Gamma^{\{F_n\}_{n\in\mathbb{N}}}_\varepsilon(y,\rho_{c})$ by (\ref{finite inv}).

		Combined with the definition of \(\overline{D}_{\mathrm{ext}}(Y_E,\varepsilon,\rho_{c})\) and the above item (2), one has
		\[
		\overline{D}_{\mathrm{ext}}(Y_E,\varepsilon,\rho_{c})
		\geq
		\overline{\mathrm{mdim}}_{M}
		\left(
		Y_{E\cup \{-m,-m+1,\cdots,m\}},
		\{F_n\}_{n\in\mathbb N},\rho_{c}
		\right)
		=
		\overline{\mathrm{mdim}}_{M}
		\left(
		Y_E,
		\{F_n\}_{n\in\mathbb N},\rho_{c}
		\right).
		\]
		Moreover, by (\ref{finite inv}) and the definition of $\Gamma_{\mathrm{as}}(y,\rho_{c})$ we have for every $y\in Y$ and any $k\geq 0$,
		$$
		\sigma^{-k}\Gamma_{\varepsilon}^{\{F_n\}_{n\in\mathbb{N}}}(y,\rho_{c})\cap \Gamma_{\mathrm{as}}(y,\rho_{c})\supset \left\{z\in Y_E:y^{(j)}=z^{(j)}\text{ for all }j\geq k-m\right\}.
		$$
		This implies that for all $k>m$
		\begin{equation}\label{estimation for as}
			\begin{aligned}
				& \ s_{F_{k}}(\sigma^{-k}\Gamma_{\varepsilon}^{\{F_n\}_{n\in\mathbb{N}}}(y,\rho_{c})\cap\Gamma_{\mathrm{as}}(y,\rho_{c}),\delta,\rho_{c}) \\
				\geq & \ s_{F_k}(\left\{z\in Y_E:y^{(j)}=z^{(j)}\text{ for all }j\geq k-m\right\},\delta,\rho_{c})	\\
				= & \ s_{F_k}(\left\{z\in Y_E:y^{(j)}=z^{(j)}\text{ for all }j\geq k-m\right\},\delta^{c},\rho) \geq s_{\{0\}}(\Lambda_{E},\delta^{c},d)^{|F_{k-m}|}.
			\end{aligned}
		\end{equation}
		Then together with (\ref{sp for YE}), one has
		$$
		\min \left\{\overline{D}_{\mathrm{ext}}^{\mathrm{pre}}(Y_{E},\varepsilon,\rho_{c}),\overline{D}_{\mathrm{ext}}^{\mathrm{as}}(Y_{E},\rho_{c})\right\}
		\geq		
		\overline{\mathrm{mdim}}_M(Y_{E},\{F_{n}\}_{n\in\mathbb{N}},\rho_{c}).
		$$
		This completes the proof.
	\end{proof}
	
	As a direct corollary of Proposition \ref{YE}, we have:
	
	\begin{prop} \label{YE-result}
		For arbitrarily given $0\le a\le b$ and $c\geq \max\{1,b\}$, one can choose  a subset
		$E\subset \mathbb{Z}$ such that for every $\varepsilon>0$,
		\[
		\begin{aligned}
			c\ \ge\ & \ \overline{D}_{\mathrm{int}}(Y_{E},\varepsilon,\rho_{c})
			=\overline{D}_{\mathrm{int}}^{\mathrm{pre}}(Y_{E},\varepsilon,\rho_{c})
			=\overline{D}_{\mathrm{int}}^{\mathrm{as}}(Y_{E},\rho_{c}) \\
			=\  & \ \overline{D}_{\mathrm{ext}}(Y_{E},\varepsilon,\rho_{c})=\overline{D}^{\mathrm{pre}}_{\mathrm{ext}}(Y_{E},\varepsilon,\rho_{c})
			= \overline{D}^{\mathrm{as}}_{\mathrm{ext}}(Y_{E},\rho_{c})= b \\
			\ge\  a\ =\ & \
			\underline{D}^{\mathrm{pre}}_{\mathrm{ext}}(Y_{E},\varepsilon,\rho_{c})
			=\underline{D}^{\mathrm{as}}_{\mathrm{ext}}(Y_{E},\rho_{c})=\underline{D}_{\mathrm{ext}}(Y_{E},\varepsilon,\rho_{c}) \\
			=\ & \ \underline{D}_{\mathrm{int}}(Y_{E},\varepsilon,\rho_{c})=
			\underline{D}_{\mathrm{int}}^{\mathrm{pre}}(Y_{E},\varepsilon,\rho_{c}) =\underline{D}_{\mathrm{int}}^{\mathrm{as}}(Y_{E},\rho_{c}).
		\end{aligned}
		\]
	\end{prop}
	
	\subsection{Separation via the product construction}\
	
	Now fix \(0\leq a\leq b\) as in Theorem~\ref{count}.
	Choose \(c\geq \max\{b,1\}\), and let \(E\subset \mathbb Z\) be given by
	Proposition~\ref{YE-result}. For simplicity, write \(Y=Y_E\), equipped with the
	fixed metric \(\rho_c\).
	
	We consider the product system $(X\times Y,\sigma\times \sigma)$ constructed as before, with the metric
	\[
	(d\times \rho_{c})\bigl((x,y),(x',y')\bigr)=\max\{d(x,x'),\rho_{c}(y,y')\}.
	\]
	From now on, all stable sets are understood with respect to these
	metrics, so we omit the metric from the notation.
	Then, for every $(x,y)\in X\times Y$ and each $i\in\mathbb{Z}$, one has
	\[\Gamma_{\mathrm{as}}((x,y))
	=
	\Gamma_{\mathrm{as}}(x)\times \Gamma_{\mathrm{as}}(y)
	\ \text{and}\
	(\sigma\times \sigma)^{i}\Gamma_{\varepsilon}^{\{F_n\}_{n\in\mathbb{N}}}((x,y))
	=
	\sigma^{i}\Gamma_{\varepsilon}^{\{F_n\}_{n\in\mathbb{N}}}(x)\times \sigma^{i}\Gamma_{\varepsilon}^{\{F_n\}_{n\in\mathbb{N}}}(y).
	\]
	
	As shown by the following lemma, the product system can be treated componentwise.

	\begin{lemma}\label{product-local-entropy}
		Let $\delta>0$, $\varepsilon\in(0,1)$ and $x\in X$. Then:
		\begin{enumerate}

			\item For every $y\in Y$, one has
			\begin{equation*}\label{less than}
				\begin{aligned}
					&h_{\mathrm{top}}(\Gamma_{\varepsilon}^{\{F_n\}_{n\in\mathbb{N}}}((x, y)),
					\delta,\{F_n\}_{n\in\mathbb{N}},d\times \rho_{c})
					\\&\quad\quad\quad\quad\quad\quad\quad\le h_{\mathrm{top}}(\Gamma_{\varepsilon}^{\{F_n\}_{n\in\mathbb{N}}}(x),
					\frac{\delta}{2},\{F_n\}_{n\in\mathbb{N}},d)
					+ h_{\mathrm{top}}(\Gamma^{\{F_n\}_{n\in\mathbb{N}}}_{\varepsilon}(y),
					\frac{\delta}{2},\{F_n\}_{n\in\mathbb{N}},\rho_{c}),
					\\&\Phi_{\varepsilon}\left((x,y),\delta,d\times\rho_{c}\right)\leq \Phi_{\varepsilon}\left(x,\frac{\delta}{2},d\right)+\Phi_{\varepsilon}\left(y,\frac{\delta}{2},\rho_{c}\right),
					\\&h_{\mathrm{top}}(\Gamma_{\mathrm{as}}((x, y)),
					\delta,\{F_n\}_{n\in\mathbb{N}},d\times \rho_{c}) \\
					&\quad\quad\quad\quad\quad\quad\quad
					\le h_{\mathrm{top}}(\Gamma_{\mathrm{as}}(x),
					\frac{\delta}{2},\{F_n\}_{n\in\mathbb{N}},d)
					+ h_{\mathrm{top}}(\Gamma_{\mathrm{as}}(y),
					\frac{\delta}{2},\{F_n\}_{n\in\mathbb{N}},\rho_{c}).
				\end{aligned}
			\end{equation*}
			
			\item Let \(y_\varepsilon=(y_\varepsilon^{(n)})_{n\in\mathbb Z}\in Y_E\) be the point
			constructed in Proposition~\ref{YE}~(3), namely,
			$
			(y_\varepsilon^{(n)})_i=0$ for all $n\in\mathbb Z$ and $|i|\leq m_{\varepsilon,c}
			:= \max\left\{
			\left\lceil
			\frac{c\cdot\log\frac{1}{\varepsilon}}{\log 2}
			\right\rceil,0
			\right\}.
			$
			Then, for any \(\delta>0\),
			\[
			\begin{aligned}
				&h_{\mathrm{top}}\bigl(
				\Gamma_{\varepsilon}^{\{F_n\}_{n\in\mathbb N}}((x,y_\varepsilon)),
				\delta,\{F_n\}_{n\in\mathbb N},d\times\rho_c
				\bigr) \\
				&\quad\quad\quad\quad\quad\quad\quad\geq
				h_{\mathrm{top}}\bigl(
				\Gamma_{\varepsilon}^{\{F_n\}_{n\in\mathbb N}}(x),
				\delta,\{F_n\}_{n\in\mathbb N},d
				\bigr)
				+
				h_{\mathrm{top}}\bigl(
				Y_{\widetilde E},
				2^{\frac{1}{c}}\delta,\{F_n\}_{n\in\mathbb N},\rho_c
				\bigr)
			\end{aligned}
			\]
			where $\widetilde E
			:=
			E\cup\{-m_{\varepsilon,c},-m_{\varepsilon,c}+1,\ldots,m_{\varepsilon,c}\}
			$,
			and for every $y\in Y$,
			\[
			\Phi_{\varepsilon}\left((x,y),\delta,d\times\rho_{c}\right)\geq \Phi_{\varepsilon}\left(x,\delta,d\right)+h_{\mathrm{top}}(Y_{E},2^{\frac{1}{c}}\delta,\{F_n\}_{n\in\mathbb{N}},\rho_{c}),
			\]
			\[
			\begin{aligned}
				&h_{\mathrm{top}}(\Gamma_{\mathrm{as}}((x, y),
				\delta,\{F_n\}_{n\in\mathbb{N}},d\times \rho_{c})  \\
				&\quad\quad\quad\quad\quad\quad\quad
				\ge h_{\mathrm{top}}(\Gamma_{\mathrm{as}}(x),\delta,\{F_n\}_{n\in\mathbb{N}},d)
				+h_{\mathrm{top}}(Y_{E},2^{\frac{1}{c}}\delta,\{F_n\}_{n\in\mathbb{N}},\rho_{c}).
			\end{aligned}
			\]
		\end{enumerate}
	\end{lemma}
	
	\begin{proof}
		(1) The conclusions follow directly from \eqref{se-sp} and that for every $A\subset X$ and $B\subset Y$,
	\[
		r_{F_n}(A\times B,\frac{\delta}{2},d\times \rho_{c})
			\leq 	r_{F_n}(A,\frac{\delta}{2},d)\cdot r_{F_n}(B,\frac{\delta}{2},\rho_{c}).
		\]
		
		(2)  Applying (\ref{se-sp-YE}) to $\widetilde{E}$ together with the fact $
		s_F(Y_{\widetilde{E}},\delta,\rho_c)
		=
		s_F(Y_{\widetilde{E}},\delta^c,\rho)
		$, we obtain
		\begin{equation}\label{gap}
			\limsup_{n\to\infty}\frac{\log s_{F_n}(Y_{\widetilde{E}},2^{\frac{1}{c}}\delta,\rho_{c})}{|F_n|}
			\le
			\log s_0(\Lambda_{\widetilde{E}},\delta^{c},d)
			\le
			\liminf_{n\to\infty}\frac{\log s_{F_n}(Y_{\widetilde{E}},\delta,\rho_{c})}{|F_n|}.
		\end{equation}

		Now using the product structure of stable sets, we obtain
			\[
			s_{F_n}((\sigma\times \sigma)^{i} \Gamma^{\{F_m\}_{m\in\mathbb{N}}}_{\varepsilon}((x,y)),\delta,d\times \rho_{c})\geq
			s_{F_n}(\sigma^{i} \Gamma^{\{F_m\}_{m\in\mathbb{N}}}_{\varepsilon}(x),\delta,d)\cdot
			s_{F_n}(\sigma^{i} \Gamma^{\{F_m\}_{m\in\mathbb{N}}}_{\varepsilon}(y),\delta,\rho_{c})
			\]
			for every $y\in Y$ and each $i\in \mathbb{Z}$.
			Hence for every $y\in Y$,
			\[
			\begin{aligned}
				\limsup_{n\to\infty}
				&\frac{1}{|F_n|}
				\log s_{F_n}
				\left((\sigma\times \sigma)^{- n} \Gamma^{\{F_m\}_{m\in\mathbb{N}}}_{\varepsilon}((x,y)),
				\delta,
				d\times \rho_{c}
				\right)
				\\
				&\ge
				\limsup_{n\to\infty}
				\left[
				\frac{
					\log s_{F_n}
					\left(
					\sigma^{-n}\Gamma_{\varepsilon}^{\{F_m\}_{m\in\mathbb N}}(x),
					\delta,
					d
					\right)
				}{|F_n|}
				+
				\frac{
					\log s_{F_n}
					\left(
					\sigma^{-n}\Gamma_{\varepsilon}^{\{F_m\}_{m\in\mathbb N}}(y),
					\delta,
					\rho_{c}
					\right)
				}{|F_n|}
				\right]
				\\
				&\ge
				\limsup_{n\to\infty}
				\frac{
					\log s_{F_n}
					\left(
					\sigma^{-n}\Gamma_{\varepsilon}^{\{F_m\}_{m\in\mathbb N}}(x),
					\delta,
					d
					\right)
				}{|F_n|}
				+
				\liminf_{n\to\infty}
				\frac{
					\log s_{F_n}
					\left(
					\sigma^{-n}\Gamma_{\varepsilon}^{\{F_m\}_{m\in\mathbb N}}(y),
					\delta,
					\rho_{c}
					\right)
				}{|F_n|}
				\\
				&\ge
				\limsup_{n\to\infty}
				\frac{
					\log s_{F_n}
					\left(
					\sigma^{-n}\Gamma_{\varepsilon}^{\{F_m\}_{m\in\mathbb N}}(x),
					\delta,
					d
					\right)
				}{|F_n|}
				+
				\log s_{0}(\Lambda_{E},\delta^{c},d)
				\quad
				(\text{by }\eqref{estimation for as})
				\\
				&\ge
				\limsup_{n\to\infty}
				\frac{
					\log s_{F_n}
					\left(
					\sigma^{-n}\Gamma_{\varepsilon}^{\{F_m\}_{m\in\mathbb N}}(x),
					\delta,
					d
					\right)
				}{|F_n|}
				+h_{\mathrm{top}}(Y_{E},2\delta^{c},\{F_n\}_{n\in\mathbb{N}},\rho)
				\quad
				(\text{by }\eqref{se-sp-YE})
				\\&=\limsup_{n\to\infty}
				\frac{
					\log s_{F_n}
					\left(
					\sigma^{-n}\Gamma_{\varepsilon}^{\{F_m\}_{m\in\mathbb N}}(x),
					\delta,
					d
					\right)
				}{|F_n|}
				+h_{\mathrm{top}}(Y_{E},2^{\frac{1}{c}}\delta,\{F_n\}_{n\in\mathbb{N}},\rho_{c}),
				\quad
				(\text{by }\eqref{newmetric})
			\end{aligned}
			\]
and
	\begin{eqnarray}
& & h_{\mathrm{top}}(\Gamma_{\mathrm{as}} ((x, y)), \delta,\{F_n\}_{n\in\mathbb{N}},d\times \rho_{c})
			\nonumber \\
&\ge & \limsup_{n\to\infty}\left[
				\frac{\log s_{F_n}(\Gamma_{\mathrm{as}}(x),\delta,d)}{|F_n|}
				+ \frac{\log s_{F_n}(\Gamma_{\mathrm{as}}(y),\delta,\rho_{c})}{|F_n|}
				\right]
			 \nonumber \\
				&\ge &
				h_{\mathrm{top}}(\Gamma_{\mathrm{as}}(x),\delta,\{F_n\}_{n\in\mathbb{N}},d)
				+ \liminf_{n\to\infty}
				\frac{\log s_{F_n}(\Gamma_{\mathrm{as}}(y),\delta,\rho_{c})}{|F_n|}
			\nonumber \\
				&\ge &
				h_{\mathrm{top}}(\Gamma_{\mathrm{as}}(x),\delta,\{F_n\}_{n\in\mathbb{N}},d)
				+ h_{\mathrm{top}}(Y_E,2^{\frac{1}{c}}\delta,\{F_n\}_{n\in\mathbb{N}},\rho_c), \label{need}
			\end{eqnarray}
where \eqref{need} again comes from \eqref{estimation for as}, \eqref{se-sp-YE} and \eqref{newmetric}.
Moreover, for \(y_\varepsilon=(y_\varepsilon^{(n)})_{n\in\mathbb Z}\in Y_E\),
			\[
			\begin{aligned}
				& h_{\mathrm{top}}(\Gamma^{\{F_n\}_{n\in\mathbb{N}}}_{\varepsilon}((x,y_\varepsilon)),
				\delta,\{F_n\}_{n\in\mathbb{N}},d\times \rho_{c})
				\\ \ge &
				\limsup_{n\to\infty}\left[
				\frac{\log s_{F_n}(\Gamma^{\{F_m\}_{m\in\mathbb{N}}}_{\varepsilon}(x),\delta,d)}{|F_n|}
				+
				\frac{\log s_{F_n}(\Gamma^{\{F_m\}_{m\in\mathbb{N}}}_{\varepsilon}(y_\varepsilon),\delta,\rho_{c})}{|F_n|}
				\right]
				\\
				\ge &
				h_{\mathrm{top}}(\Gamma^{\{F_n\}_{n\in\mathbb{N}}}_{\varepsilon}(x),\delta,\{F_n\}_{n\in\mathbb{N}},d)
				+
				\liminf_{n\to\infty}
				\frac{\log s_{F_n}(\Gamma^{\{F_m\}_{m\in\mathbb{N}}}_{\varepsilon}(y_\varepsilon),\delta,\rho_{c})}{|F_n|}
				\\
				\ge &
				h_{\mathrm{top}}(\Gamma^{\{F_n\}_{n\in\mathbb{N}}}_{\varepsilon}(x),\delta,\{F_n\}_{n\in\mathbb{N}},d)
				+
				\liminf_{n\to\infty}
				\frac{\log s_{F_n}(Y_{\widetilde{E}},\delta,\rho_{c})}{|F_n|}\quad (\text{by }\eqref{stable-YE})
				\\
				\ge &
				h_{\mathrm{top}}(\Gamma^{\{F_n\}_{n\in\mathbb{N}}}_{\varepsilon}(x),\delta,\{F_n\}_{n\in\mathbb{N}},d)
				+
				\limsup_{n\to\infty}
				\frac{\log s_{F_n}(Y_{\widetilde{E}},2^{\frac{1}{c}}\delta,\rho_{c})}{|F_n|}\quad (\text{by }\eqref{gap})
				\\
				= &
				h_{\mathrm{top}}(\Gamma^{\{F_n\}_{n\in\mathbb{N}}}_{\varepsilon}(x),\delta,\{F_n\}_{n\in\mathbb{N}},d)
				+
				h_{\mathrm{top}}(Y_{\widetilde{E}},2^{\frac{1}{c}}\delta,\{F_n\}_{n\in\mathbb{N}},\rho_{c}).
			\end{aligned}
			\]
		This ends the proof.
	\end{proof}
	
	Obviously Theorem \ref{count} comes from the result below.

	\begin{prop} \label{answer}
		Let \((X\times Y,\sigma\times\sigma)\) be constructed as above and \(\varepsilon\in(0,1)\). Then
		\[
		\begin{aligned}
			&\overline{D}_{\mathrm{int}}(X\times Y,\varepsilon,d\times \rho_{c})=\overline{D}_{\mathrm{int}}^{\mathrm{pre}}(X\times Y,\varepsilon,d\times \rho_{c})=\overline{D}_{\mathrm{int}}^{\mathrm{as}}(X\times Y,d\times \rho_{c})=1+b,
			\\&\underline{D}_{\mathrm{int}}(X\times Y,\varepsilon,d\times \rho_{c})=\underline{D}_{\mathrm{int}}^{\mathrm{pre}}(X\times Y,\varepsilon,d\times \rho_{c})=\underline{D}_{\mathrm{int}}^{\mathrm{as}}(X\times Y,d\times \rho_{c})=1+a,
			\\	&\overline{D}_{\mathrm{ext}}(X\times Y,\varepsilon,d\times \rho_{c})=\overline{D}_{\mathrm{ext}}^{\mathrm{pre}}(X\times Y,\varepsilon,d\times \rho_{c})=\overline{D}_{\mathrm{ext}}^{\mathrm{as}}(X\times Y,d\times \rho_{c})=b,
			\\&\underline{D}_{\mathrm{ext}}(X\times Y,\varepsilon,d\times \rho_{c})=\underline{D}_{\mathrm{ext}}^{\mathrm{pre}}(X\times Y,\varepsilon,d\times \rho_{c})=\underline{D}_{\mathrm{ext}}^{\mathrm{as}}(X\times Y,d\times \rho_{c})=a.
		\end{aligned}
		\]
	\end{prop}
	\begin{proof}
		Firstly, let us prove $\overline{D}_{\mathrm{int}}(X\times Y,\varepsilon,d\times\rho_{c})=1+b$. Lemma \ref{product-local-entropy} (2) gives
		\[
		\begin{aligned}
			\overline{D}_{\mathrm{int}}(X&\times Y_{E},\varepsilon,d\times \rho_{c})
			\\	&\geq
			\limsup_{\delta \to 0}
			\frac{
				\sup\limits_{x \in X}
				h_{\mathrm{top}}(\Gamma^{\{F_n\}_{n\in\mathbb{N}}}_{\varepsilon}(x),\delta,\{F_n\}_{n\in\mathbb{N}},d)
				+
				h_{\mathrm{top}}(Y_{\widetilde{E}},2^{\frac{1}{c}}\delta,\{F_n\}_{n\in\mathbb{N}},\rho_{c})
			}{
				\log \frac{1}{\delta}
			}
			\\
			&\geq
			\liminf_{\delta\to 0}
			\frac{
				\sup\limits_{x \in X}
				h_{\mathrm{top}}(\Gamma^{\{F_n\}_{n\in\mathbb{N}}}_{\varepsilon}(x),\delta,\{F_n\}_{n\in\mathbb{N}},d)
			}{
				\log\frac{1}{\delta}
			}
			+
			\limsup_{\delta\to 0}
			\frac{
				h_{\mathrm{top}}(Y_{\widetilde{E}},2^{\frac{1}{c}}\delta,\{F_n\}_{n\in\mathbb{N}},\rho_{c})
			}{
				\log \frac{1}{\delta}
			}
			\\
			&=
			\underline{D}_{\mathrm{int}}(X,\varepsilon,d)
			+
			\overline{\mathrm{mdim}}_M(Y_{\widetilde{E}},\{F_{n}\}_{n\in\mathbb{N}},\rho_{c})
			=1+b,
		\end{aligned}
		\]
		where the last equality is due to Proposition \ref{YE} (2).
		The reverse inequality comes from Lemma \ref{product-local-entropy} (1) obviously:
		\(
		\overline{D}_{\mathrm{int}}(X\times Y,\varepsilon,d\times \rho_{c})
		\leq
		\overline{D}_{\mathrm{int}}(X,\varepsilon,d)
		+
		\overline{D}_{\mathrm{int}}(Y,\varepsilon,\rho_{c})
		=1+b.
		\)

		All of the remaining three identities follow similarly from Lemma \ref{product-local-entropy} as below:
		\[
		\underline{D}_{\mathrm{int}}(X\times Y,\varepsilon,d\times \rho_{c})
		\geq
		\underline{D}_{\mathrm{int}}(X,\varepsilon,d)
		+	\underline{\mathrm{mdim}}_M(Y,\{F_{n}\}_{n\in\mathbb{N}},\rho_{c})
		=1+a,
		\]
		\[
		\underline{D}_{\mathrm{int}}(X\times Y,\varepsilon,d\times \rho_{c}) \leq \overline{D}_{\mathrm{int}}(X,\varepsilon,d)
		+
		\underline{D}_{\mathrm{int}}(Y,\varepsilon,\rho_{c})
		=1+a.
		\]
		\[
		\overline{D}_{\mathrm{ext}}(X\times Y,\varepsilon,d\times \rho_{c})
		\geq \underline{D}_{\mathrm{ext}}(X,\varepsilon,d)
		+
		\overline{\mathrm{mdim}}_M(Y,\{F_{n}\}_{n\in\mathbb{N}},\rho_{c})
		=b,
		\]
		\[
		\overline{D}_{\mathrm{ext}}(X\times Y,\varepsilon,d\times \rho_{c})
		\leq
		\overline{D}_{\mathrm{ext}}(X,\varepsilon,d)
		+
		\overline{D}_{\mathrm{ext}}(Y,\varepsilon,\rho_{c})
		=b.
		\]
		\[
		\underline{D}_{\mathrm{ext}}(X\times Y,\varepsilon,d\times \rho_{c})
		\geq
		\underline{D}_{\mathrm{ext}}(X,\varepsilon,d)
		+
		\underline{\mathrm{mdim}}_M(Y,\{F_{n}\}_{n\in\mathbb{N}},\rho_{c})
		=a,
		\]
		\[
		\underline{D}_{\mathrm{ext}}(X\times Y,\varepsilon,d\times \rho_{c})
		\leq \overline{D}_{\mathrm{ext}}(X,\varepsilon,d)
		+
		\underline{D}_{\mathrm{ext}}(Y,\varepsilon,\rho_{c})
		=a.
		\]
	The identities for the preimage and asymptotic versions follow in the same way from the corresponding estimates in Lemma~\ref{product-local-entropy}.
			This completes the proof.
	\end{proof}
	
	As a direct consequence, one has (compared with \eqref{relation}):
	
	\begin{remark} \label{last}
		By choosing suitable metric $\rho_{c}$ and suitable subsets \(E_1,E_2\subset \mathbb Z\), we have
		\begin{eqnarray*}
			\overline{D}_{\mathrm{int}}(X\times Y_{E_1},\varepsilon,d\times \rho_{c})
			& > & \underline{D}_{\mathrm{int}}(X\times Y_{E_1},\varepsilon,d\times \rho_{c}) \\
			& > & \overline{D}_{\mathrm{ext}}(X\times Y_{E_1},\varepsilon,d\times \rho_{c})
			>
			\underline{D}_{\mathrm{ext}}(X\times Y_{E_1},\varepsilon,d\times \rho_{c})
		\end{eqnarray*}
		and
		\begin{eqnarray*}
			\overline{D}_{\mathrm{int}}(X\times Y_{E_2},\varepsilon,d\times \rho_{c})
			& > & 			\overline{D}_{\mathrm{ext}}(X\times Y_{E_2},\varepsilon,d\times \rho_{c}) \\
			& > & 			\underline{D}_{\mathrm{int}}(X\times Y_{E_2},\varepsilon,d\times \rho_{c})
			>
			\underline{D}_{\mathrm{ext}}(X\times Y_{E_2},\varepsilon,d\times \rho_{c}).
		\end{eqnarray*}
		for every
		\(\varepsilon\in(0,1)\).
		The corresponding preimage and asymptotic versions also hold.
	\end{remark}

	\section*{Acknowledgements}
	
	We would like to thank Professors Ercai Chen, Dou Dou, Rui Yang and Xiaoyao Zhou for useful discussions. Guohua Zhang is supported by the National Key Research and Development Program of China (No. 2021YFA1003204).
	Ruifeng Zhang is supported by National Natural Science Foundation of China (No. 12031019).

	\bibliographystyle{alpha}
	
	
	

	
\end{document}